\documentclass[
a4paper,
11pt,
DIV=15,
toc=bibliography, 
headsepline, 
parskip=half-,
abstract=true,
]{scrartcl}
  
\usepackage[english]{babel}
\usepackage[utf8]{inputenc}
\usepackage[T1]{fontenc}
\usepackage[english,cleanlook]{isodate}
\usepackage{lmodern}
\usepackage[babel=true]{microtype}
\usepackage{amsmath}
\usepackage{mathtools}
\usepackage{amssymb}
\usepackage{amsthm}
\usepackage{scrlayer-scrpage} 
\usepackage{enumitem}
\usepackage{tabularx}
\usepackage{mathrsfs}
\usepackage{tikz} 

\title{Noncompact self-shrinkers for mean curvature flow with arbitrary genus}
\author{Reto Buzano, Huy The Nguyen and Mario B. Schulz}
\date{\vspace*{-4ex}}

\ohead{Noncompact self-shrinkers for mean curvature flow with arbitrary genus}
\ihead{R.\;Buzano, H.\;Nguyen, M.\;Schulz}
\setkomafont{pageheadfoot}{\footnotesize}

\providecommand{\R}{\mathbb{R}}
\providecommand{\N}{\mathbb{N}}
\providecommand{\Z}{\mathbb{Z}}
\providecommand{\B}{B}
\providecommand{\Sp}{\mathbb{S}}
\providecommand{\dih}{\mathbb{D}}
\providecommand{\hsd}{\mathscr{H}}
\providecommand{\ess}{\mathrm{e}}
\providecommand{\symdiff}{\mathbin{\triangle}}

\DeclarePairedDelimiter\abs{\lvert}{\rvert}

\DeclarePairedDelimiter\sk{\langle}{\rangle}
\DeclarePairedDelimiter\interval{]}{[}
\DeclarePairedDelimiter\Interval{[}{[}

\DeclareMathOperator{\genus}{genus}

\DeclareMathOperator{\dist}{dist}

\usepackage[
pdftitle={Noncompact self-shrinkers for mean curvature flow with arbitrary genus},
pdfauthor={Reto Buzano, Huy The Nguyen and Mario B. Schulz}, 
hidelinks,
]{hyperref}

\theoremstyle{plain}
\newtheorem{theorem}{Theorem}[section]
\newtheorem{lemma}[theorem]{Lemma}
\newtheorem{corollary}[theorem]{Corollary}

\newtheorem{conjecture}[theorem]{Conjecture}

\theoremstyle{definition}
\newtheorem{definition}[theorem]{Definition}

\theoremstyle{remark}

\makeatletter
\renewenvironment{proof}[1][\proofname]{\par
  \pushQED{\qed}%
  \normalfont \topsep0\p@\@plus6\p@\relax
  \trivlist
  \item[\hskip\labelsep
        \itshape
    #1\@addpunct{.}]\ignorespaces
}{%
  \popQED\endtrivlist\@endpefalse
}
\makeatother

\newcommand\printaddress{{
\setlength{\parindent}{15pt}
\setlength{\parskip}{2.5ex}
\small~
\par
{\scshape Reto Buzano}
\newline 
Universit\`a degli Studi di Torino, 
Dipartimento di Matematica,
Via Carlo Alberto 10, 
10123 Torino, Italy 
\newline
\textit{E-mail address:} 
\texttt{reto.buzano@unito.it}
\par
{\scshape Huy The Nguyen}
\newline 
Queen Mary University of London, 
School of Mathematical Sciences, 
Mile End Road, 
London E1 4NS, UK
\newline
\textit{E-mail address:} 
\texttt{h.nguyen@qmul.ac.uk}
\par
{\scshape Mario B. Schulz}
\newline 
University of M\"unster, 
Mathematisches Institut, 
Einsteinstrasse 62,
48149 M\"unster,
Germany
\newline
\textit{E-mail address:} 
\texttt{m.schulz@uni-muenster.de}
\par
}} 

\begin{document}
 
\maketitle

\begin{abstract} 
In his lecture notes on mean curvature flow, Ilmanen conjectured the existence of noncompact self-shrinkers with arbitrary genus. 
Here, we employ min-max techniques to give a rigorous existence proof for these surfaces. 
Conjecturally, the self-shrinkers that we obtain have precisely one (asymptotically conical) end. 
We confirm this for large genus via a precise analysis of the limiting object of sequences of such self-shrinkers for which the genus tends to infinity. 
Finally, we provide numerical evidence for a further family of noncompact self-shrinkers with odd genus and two asymptotically conical ends.
\end{abstract}

\section{Introduction}\label{sec:Intro}

A mean curvature flow starting from a closed embedded surface must necessarily form a singularity in finite time. 
The formation of singularities is therefore one of the central themes in the study of mean curvature flow.
By work of Huisken \cite{Huisken1990}, Ilmanen \cite{Ilmanen1995}, and White \cite{White1994}, such singularities are modelled on surfaces which shrink self-similarly along the flow. 
More precisely, parabolic rescalings around a singularity of mean curvature flow subsequentially converge to a self-shrinker.  

\begin{definition}\label{defn:self-shrinker}
A surface $\Sigma\subset\R^{3}$ with unit normal vector field $\nu$ is called \emph{self-shrinker} for mean curvature flow if one of the following equivalent conditions is satisfied. 
\begin{enumerate}[label={(\roman*)}]
\item\label{defn:self-shrinker-1} The family $\{\Sigma_t\}_{t\in\interval{-\infty,0}}$ of surfaces $\Sigma_t\vcentcolon=\sqrt{-t}\,\Sigma$ is a mean curvature flow, i.\,e.~the normal part of the derivative with respect to $t$ is equal to the mean curvature vector. 
\item\label{defn:self-shrinker-2} The mean curvature $H=\operatorname{div}_{\Sigma}\nu$ of $\Sigma$ satisfies the equation
\begin{align}\label{eqn:self-shrinker}
H=\tfrac{1}{2}\sk{x,\nu}
\end{align}
at each $x\in\Sigma$, where $\sk{\cdot,\cdot}$ denotes the Euclidean scalar product in $\R^3$. 
\item\label{defn:self-shrinker-3} $\Sigma$ is a critical point for the \emph{Gaußian area functional} 
\begin{align}\label{eqn:functional}
F(\Sigma)&\vcentcolon=\frac{1}{4\pi}\int_{\Sigma}e^{-\frac{1}{4}\abs{x}^2}\,d\hsd^{2}(x) 
\end{align}
where $\hsd^2$ denotes the $2$-dimensional Hausdorff measure in $\R^3$.
\item\label{defn:self-shrinker-4} $\Sigma$ is a minimal surface in the (conformally flat) Riemannian manifold $\bigl(\R^3,e^{-\frac{1}{4}\abs{x}^2}g_{\R^3}\bigr)$ called \emph{Gaußian space}.  
\end{enumerate}
\end{definition}

References for the equivalence of \ref{defn:self-shrinker-1}--\ref{defn:self-shrinker-4} are results by  
Huisken~\cite[§\,3]{Huisken1990} and the work of 
Colding and Minicozzi \cite[§\,2--3]{Colding2012a}, \cite[§\,1]{Colding2012}. 

\pagebreak[2]

Noncompact self-shrinkers model local singularities along the mean curvature flow, i.\,e.~situations where the surface does not become extinct at the singular time. 
Therefore, a very natural objective in this setting is to construct examples of complete, embedded self-shrinkers.  
The simplest ones are flat planes, the sphere of radius $2$ and cylinders of radius $\sqrt{2}$, all centred at the origin. 
Brendle~\cite{Brendle2016} proved that there are no other embedded self-shrinkers of genus zero. 
Angenent~\cite{Angenent1992} constructed a rotationally symmetric, closed self-shrinker of genus one. 
Until now, Angenent's torus has been the only known example of an embedded self-shrinker of genus one.  
Based on numerical simulations using Brakke's \cite{Brakke1992} surface evolver, Ilmanen \cite{Ilmanen1998} conjectured the existence of noncompact embedded self-shrinkers with dihedral symmetry, one end and \emph{arbitrary genus}. 
For high genus, these surfaces resemble the union of a sphere and a plane desingularised along the line of intersection. 
Kapouleas, Kleene and M{\o}ller \cite{Kapouleas2018} as well as X.~Nguyen \cite{Nguyen2014} were able to formalise the desingularisation procedure and prove the existence of such self-shrinkers if the genus is sufficiently large. 
However, the nature of the desingularisation method does not allow the construction of low genus examples. 
Our main result establishes the existence of self-shrinkers with arbitrary genus. 

\begin{theorem}\label{thm:main}
For all $g\in\N$ there exists a complete, embedded, noncompact self-shrinker $\Theta_g\subset\R^3$ for mean curvature flow which has genus $g$ and is invariant under the action of the dihedral group~$\dih_{g+1}$.
\end{theorem}

In particular, $\Theta_1$ is the first example of a \emph{noncompact} self-shrinker with genus one and 
$\Theta_2$ is the very first example of a self-shrinker with genus two. 
Based on the numerical evidence presented in Section~\ref{sec:visualisation}, we conjecture that $\Theta_1$ has less Gaußian area than Angenent's torus. 
We also conjecture that $\Theta_1$ has Morse index equal to $5$. 

By a result of L.~Wang \cite{Wang2016}, a complete, embedded self-shrinker can only have asymptotically conical or cylindrical ends.
By construction, the self-shrinkers $\Theta_g$ will have at least one asymptotically conical end and we conjecture that this is the only end of $\Theta_g$. 
The question whether or not $\Theta_g$ may have additional asymptotically \emph{cylindrical} ends is closely related to another conjecture of Ilmanen \cite[Lecture 3]{Ilmanen1998} stating that if a self-shrinker has an asymptotically cylindrical end then it must be isometric to the self-shrinking cylinder. 
A partial affirmative answer to this conjecture was given by L.~Wang \cite[Theorem~1.1]{Wang2016a} but even a full resolution would in theory still allow $\Theta_g$ to have additional asymptotically conical ends.
By analysing the behaviour for $g\to\infty$ we are able to rule out any additional ends for $\Theta_g$ if the genus $g$ is sufficiently large. 

\begin{theorem}\label{thm:high_genus} 
The sequence of self-shrinkers $\{\Theta_g\}_{g\in\N}$ constructed in Theorem \ref{thm:main} converges to the union of the horizontal plane and the self-shrinking sphere in the sense of varifolds as $g\to\infty$. 
The convergence is locally smooth away from the intersection circle. 
In particular, if $g$ is sufficiently large, then $\Theta_g$ has exactly one asymptotically conical end. 
\end{theorem}

We will contrast this result by providing examples of numerically obtained self-shrinkers with odd genus and \emph{two} asymptotically conical ends (see Section \ref{sec:visualisation}, Figure \ref{fig:2ends}). 
Based on this numerical evidence, we propose the following new conjecture.

\begin{conjecture}\label{conj:twoends}
For each sufficiently large $n\in\N$ there exists a complete, embedded, noncompact self-shrinker $\Phi_g\subset\R^3$ for mean curvature flow which has genus $g=2n-1$, is invariant under the action of the dihedral group $\dih_{n}$ and has two asymptotically conical ends.
These surfaces converge (locally) to the union of the self-shrinking cylinder and the sphere as $n\to\infty$. 
\end{conjecture}

Conjecture~\ref{conj:twoends} suggests that the union of a self-shrinking cylinder and a sphere can be desingularised to an embedded self-shrinker with two ends. 
In Section \ref{sec:visualisation} we compare these numerically obtained self-shrinkers with visualisations of M\o{}ller's \cite{Moeller2014} desingularisation of Angenent's torus and a sphere. 

Based on numerical approximations, Chopp \cite[§\,4.2]{Chopp1994} conjectured that the doubling of any regular polyhedron joined by  necks at the center of each face is isotopic to a closed self-shrinker. 
Later, Ketover \cite{Ketoverc} employed equivariant variational methods to obtain self-shrinkers with tetrahedral, octahedral and icosahedral symmetry, but it remains open whether these surfaces are actually closed or whether they may have additional noncompact ends.

The article is organised as follows.
Section \ref{sec:existence} contains the construction of the surfaces $\Theta_g$ and the proof of Theorem~\ref{thm:main} using equivariant min-max theory. 
In Section \ref{sec:high_genus} we analyse the behaviour as $g\to\infty$ and prove Theorem~\ref{thm:high_genus}.  
Finally, in Section \ref{sec:visualisation}, we reproduce Ilmanen's numerical simulations and include new images visualising the symmetry and geometry of our surfaces $\Theta_g$ and we also provide images of the family of self-shrinkers motivating Conjecture \ref{conj:twoends}.

The min-max theory in Gaußian space was developed by Ketover and Zhou \cite{Ketover2018} based on the pioneering work by Almgren and Pitts, and later refinements by Simon--Smith and Colding--De Lellis (cf.~\cite{Colding2003}). 
The equivariant version of min-max theory was introduced by Ketover \cite{Ketover,Ketovera,Ketoverc}. His work \cite{Ketovera} on free-boundary minimal surfaces in the Euclidean unit ball contains the remark that equivariant min-max methods can be used to construct the self-shrinkers discovered in \cite{Kapouleas2018,Nguyen2014}, however no further details are given. 
Min-max methods rely on the the choice of an effective sweepout and our sweepout described in Section~\ref{subsec:sweepout} is designed to allow good control on the Gaußian (rather than the Euclidean) area. 
In general, the convergence of a min-max sequence is only obtained in the sense of varifolds and additional work is required to control the topology of the limit surface. 
Our argument to determine the genus differs from the approach described in \cite{Ketovera} and in particular relies on our Lemma \ref{lem:genus} about the structure of arbitrary closed, equivariant surfaces, allowing a generalisation of \cite[Lemma B.1]{Carlotto2020} used in recent work of the third author in collaboration with Carlotto and Franz.
Moreover, proving the behaviour for high genus as stated in Theorem~\ref{thm:high_genus} is rather delicate and we rely again on a careful application of Lemma~\ref{lem:genus}. 

\paragraph{Acknowledgements.}  
The authors would like to thank Felix Schulze for interesting discussions about asymptotically conical self-shrinkers. Moreover, they thank the referee for valuable comments and suggestions.
The research was funded by the EPSRC grant EP/S012907/1 and 
M.\,S. was partly funded by the Deutsche Forschungsgemeinschaft (DFG, German Research Foundation) under Germany's Excellence Strategy EXC 2044 -- 390685587, Mathematics M\"unster: Dynamics--Geometry--Structure, and the Collaborative Research Centre CRC 1442, Geometry: Deformations and Rigidity.

\pagebreak[1] 
 
\section{Existence of noncompact self-shrinkers with given genus}
\label{sec:existence}

In this section, we employ an equivariant min-max procedure to construct the self-shrinkers $\Theta_g\subset\R^3$ for any given $g\in\N$ as stated in Theorem \ref{thm:main}. 
Given $2\leq n\in\N$, the \emph{dihedral group} $\dih_{n}$ of order $2n$ is defined to be the discrete subgroup of Euclidean isometries acting on $\R^3$ generated by the rotations $\psi_\ell\colon\R^3\to\R^3$ of angle $\pi$ around the $n$ horizontal axes 
\begin{align}\label{eqn:axes}
\xi_\ell\vcentcolon=\{(r\cos\tfrac{\ell\pi}{n},r\sin\tfrac{\ell\pi}{n},0)\mid r\in\R\}
\end{align}
for $\ell\in\{1,\ldots,n\}$. 
In particular, $\dih_{n}$ contains the composition $\psi_{2}\circ\psi_{1}$ which acts as a rotation by angle $\frac{2\pi}{n}$ around the vertical axis 
\begin{align}\label{eqn:axis0}
\xi_0\vcentcolon=\{(0,0,s)\mid s\in\R\}.
\end{align}
It will be clear from the proof of Theorem \ref{thm:main} that the self-shrinker $\Theta_g$ in question contains the horizontal axes $\xi_1,\ldots,\xi_{g+1}$ and intersects the vertical axis $\xi_0$ orthogonally. 
By construction, we expect that $\Theta_g$ is actually invariant under the action of the anti-prismatic symmetry group of order $4(g+1)$, meaning in particular that in addition to the invariance under the dihedral group $\dih_{g+1}$ of order $2(g+1)$ it is also invariant under reflections with respect to any plane spanned by $\xi_0$ and $\xi_{\ell+\frac{1}{2}}$ as defined in \eqref{eqn:axis0} and \eqref{eqn:axes}, but in our proofs it is more convenient to work with the subgroup of rotations.

\subsection{Construction of the sweepout}
\label{subsec:sweepout}

In the following, we recall the definitions of equivariant sweepouts, isotopies, and saturations, as well as the notion of min-max width (cf.~\cite{Ketover2018,Ketover,Ketoverc} and \cite{Carlotto2020,Franz2021}) that are needed for the equivariant min-max construction.

\begin{definition}
Let $\dih_n$ be the dihedral group for some $2\leq n\in\N$.
A family $\{\Sigma_t\}_{t\in[0,1]}$ of surfaces $\Sigma_t\subset\R^3$ with the following properties is called $\dih_n$\emph{-sweepout} of $\R^3$. 
\begin{enumerate}[label={\normalfont(\roman*)}] 
\item For all $t\in\interval{0,1}$ the set $\Sigma_t\subset\R^3$ is a smooth, embedded surface without boundary. 
\item $\Sigma_t$ varies locally smoothly for $t\in\interval{0,1}$, and continuously, in the sense of varifolds, for $t\in[0,1]$. 
\item Every $\Sigma_t$ is $\dih_n$-equivariant, i.\,e. $\varphi(\Sigma_t)=\Sigma_t$ for all $\varphi\in \dih_n$ and all $t\in[0,1]$.  
\end{enumerate}
Given $2\leq n\in\N$, a smooth function $\Phi\colon[0,1]\times\R^3\to\R^3$ is called $\dih_{n}$\emph{-isotopy} if $\Phi(t,\cdot)$ is a diffeomorphism of $\R^3$ for all $t\in[0,1]$, $\Phi(0,\cdot)$ and $\Phi(1,\cdot)$ coincide with the identity map in $\R^3$, and $\Phi(t,\varphi(\cdot))=\varphi(\Phi(t,\cdot))$ for all $t\in[0,1]$ and every $\varphi\in\dih_{n}$. 

Given a $\dih_{n}$-sweepout $\{\Sigma_t\}_{t\in[0,1]}$ of $\R^3$ we define the \emph{Gaußian min-max width} $W$ of its $\dih_n$\emph{-saturation} $\Pi\vcentcolon=\{\{\Phi(t,\Sigma_t)\}_{t\in[0,1]}\mid \Phi\text{ is a $\dih_n$-isotopy}\}$ as
\begin{align}
W&\vcentcolon=\adjustlimits\inf_{\{\Lambda_t\}\in \Pi~}\sup_{t\in[0,1]}F({\Lambda_t}).
\end{align}
\end{definition}

We construct a sweepout with the right topology and symmetry and with control on its Gaußian area. 
The idea is to sweep out $\R^3$ by rescalings of a concentric sphere desingularised with a horizontal plane. 
This differs from the construction in \cite[Lemma~2.2]{Carlotto2020} which is based on an equivariant gluing of three parallel discs through suitably controlled ribbons although both sweepouts have the same symmetry and very similar topology. 

\begin{lemma}\label{lem:sweepout}
Given any $1\leq g\in\N$ there exists a $\dih_{g+1}$-sweepout $\{\Sigma_t\}_{t\in[0,1]}$ of $\R^3$ such that 
\begin{itemize}[nosep]
\item $\Sigma_t$ has genus $g$ and one end for every $0<t<1$; 
\item $\Sigma_t$ contains the horizontal axes $\xi_1,\ldots,\xi_{g+1}$ defined in \eqref{eqn:axes} for every $t\in[0,1]$; 
\item $F(\Sigma_t)\leq 1+\frac{4}{e}+\frac{1}{50}$ for every $t\in[0,1]$, recalling the Gaußian area functional $F$ from \eqref{eqn:functional}. 
\end{itemize}
\end{lemma}
  
\begin{proof}	
For all $\ell\in\{1,\ldots,2g+2\}$ we consider the sets 
\begin{align*}
\sigma_\ell&\vcentcolon=\{(r\cos\alpha,r\sin\alpha,0)\mid \tfrac{(\ell -1)\pi}{g+1}\leq\alpha\leq\tfrac{\ell \pi}{g+1},~0\leq r<1\}, 
\\
\varsigma_\ell&\vcentcolon=\{(r\cos\alpha,r\sin\alpha,0)\mid \tfrac{\ell \pi}{g+1}\leq\alpha\leq\tfrac{(\ell +1)\pi}{g+1},~r>1\}
\end{align*}
shown in Figure \ref{fig:sectors} and the upper half sphere $\Sp^{+}\vcentcolon=\{x=(x_1,x_2,x_3)\in\R^3\mid x_3\geq0,~\abs{x}=1\}$.
\begin{figure}\centering
\pgfmathsetmacro{\genus}{3}
\begin{tikzpicture}[line cap=round,line join=round,baseline={(0,0)}]
\pgfmathsetmacro{\n}{\genus+1}
\pgfmathsetmacro{\m}{2*\genus+2}
\pgfmathsetmacro{\xmax}{\textwidth/2cm}
\clip(-\xmax,-3)rectangle(\xmax,3);
\foreach \k in {1,...,\n}{
\fill[black!40](0,0)--({\k*360/\n}:2)arc({\k*360/\n}:{(\k+1/2)*360/\n}:2)--cycle;
\fill[black!10](0,0)--({\k*360/\n}:2)arc({\k*360/\n}:{(\k-1/2)*360/\n}:2)--cycle;
\fill[black!40]({(\k-1/2)*360/\n}:2)arc({(\k-1/2)*360/\n}:{\k*360/\n}:2)--({\k*360/\n}:2*\xmax)arc({\k*360/\n}:{(\k-1/2)*360/\n}:2*\xmax);
\fill[black!10]({(\k+1/2)*360/\n}:2)arc({(\k+1/2)*360/\n}:{\k*360/\n}:2)--({\k*360/\n}:2*\xmax)arc({\k*360/\n}:{(\k+1/2)*360/\n}:2*\xmax);
}
\foreach \k in {1,...,\m}{
\draw({(\k-1/2)*180/\n}:1.33)node[]{$\sigma_{\k}$};
\draw({(\k+1/2)*180/\n}:2.66)node[]{$\varsigma_{\k}$};
}
\end{tikzpicture}
\caption{Sectors in the plane for the case $g=\genus$.}%
\label{fig:sectors}%
\end{figure}
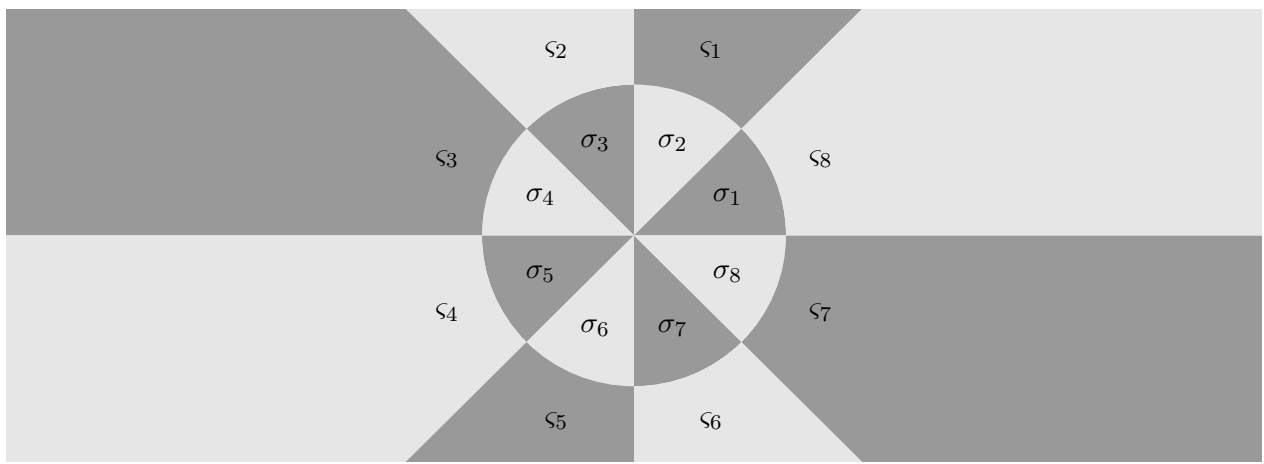
For every even $\ell$ we glue $\sigma_\ell $ and $\varsigma_\ell $ to $\Sp^+$, that is 
\begin{align*} 
\Sigma^+&\vcentcolon=\bigcup_{\ell \text{ even}}(\sigma_\ell \cup\varsigma_\ell \cup\Sp^+).
\shortintertext{For every $0<t<1$ we then set }
\Sigma_t^+&\vcentcolon=\frac{t}{1-t}\Sigma^+. 
\end{align*}
The Gaußian area of $\Sigma_t^+$ is given by half the sum of the Gaußian areas of the sphere $\Sp_r$ of radius $r=\frac{t}{1-t}$ and the plane. 
Note that the function $r\mapsto F(\Sp_r)=r^2e^{-\frac{1}{4}r^2}$ has a global maximum at $r=2$. 
Hence, for every $0<t<1$,  
\begin{align*}
2F(\Sigma_t^+)\leq 1+F(\Sp_2)=1+\tfrac{4}{e}. 
\end{align*}
There exists a $\Z_{g+1}$-equivariant pertubation $\tilde\Sigma_t^+$ of $\Sigma_t^+$, depending smoothly on $0<t<1$, such that   
\begin{itemize}[nosep]
\item the interior of $\tilde\Sigma_t^+$ is smooth, embedded and contained in the open half space $\{x_3>0\}$; 
\item the intersection of $\tilde\Sigma_t^+$ with $\{x_3=0\}$ coincides with the union $\xi_1\cup\ldots\cup\xi_{g+1}$;
\item $F(\tilde\Sigma_t^+)<F(\Sigma_t^+)+1/100$ for every $0<t<1$; 
\item $\tilde\Sigma_t^+$ and $\Sigma_t^+$ have the same limit (supported in $\{x_3=0\}$) as $t\to0$ respectively as $t\to1$.
\end{itemize}
Let $\tilde\Sigma_t^-$ be the rotation of $\tilde\Sigma_t^+$ by angle $\pi$ around $\xi_1$.  
By construction, 
\(
\Sigma_t\vcentcolon= \tilde\Sigma_t^+\cup\tilde\Sigma_t^- 
\) 
is a smooth, $\dih_{g+1}$-equivariant surface with the desired properties. 
Finally, we set $\Sigma_0$ and $\Sigma_1$ to be equal to the horizontal plane. 
\end{proof}

\subsection{Min-max procedure}\label{subsec:min-max}

Given $1\leq g\in\N$ let $\{\Sigma_t\}_{t\in[0,1]}$ be the $\dih_{g+1}$-sweepout constructed in Lemma \ref{lem:sweepout} and let $W$ be the width of its $\dih_{g+1}$-saturation $\Pi$. 
The goal is to show that the sweepout is nontrivial in the sense that $W>\max\{F(\Sigma_0),F(\Sigma_1)\}=1$. 
The argument is based on the isoperimetric inequality in Gaußian space which is formulated in terms of finite perimeter sets. 
Given any Lebesgue measurable set $E\subseteq\R^3$, we define its \emph{Gaußian measure} $\Gamma(E)$ and its \emph{Gaußian perimeter} $P(E)$ by 
\begin{align}\label{eqn:perimeter}
\Gamma(E)&\vcentcolon=(4\pi)^{-\frac{3}{2}}\int_{E}e^{-\frac{1}{4}\abs{x}^2}\,dx, & 
P(E)\vcentcolon=F(\partial^{\ess}E)=\frac{1}{4\pi}\int_{\partial^{\ess}E}e^{-\frac{1}{4}\abs{x}^2}\,d\hsd^2(x), 
\end{align}
where $\partial^{\ess}E$ denotes the \emph{essential boundary} of $E$ consisting of those points of $\R^3$ where the density of $E$ is neither $0$ nor $1$ (see \cite[§\,2]{Cianchi2011}). 
The normalisation in \eqref{eqn:perimeter} is chosen such that the upper half-space $\R^3_+=\{(x_1,x_2,x_3)\in\R^3\mid x_3>0\}$ satisfies $\Gamma(\R^3_+)=\frac{1}{2}$ and $P(\R^3_+)=1$. 
Note that \eqref{eqn:perimeter} differs from the definitions in \cite{Cianchi2011} by scaling $x\mapsto\sqrt{2}x$. 
This however does not affect the validity of the 
Gaußian isoperimetric inequality \cite{Sudakov1974,Borell1975} (see also \cite[Theorem 20]{Ros2005})  which states that among all subsets $E\subset\R^3$ with prescribed Gaußian measure $\Gamma(E)=\gamma\in[0,1]$, half-spaces, i.\,e.~sets of the form $\{x\in\R^3\mid\sk{x,\nu}>s\}$ for some $\nu\in\Sp^2$ and $s\in\R$, have the least Gaußian perimeter.

A set $E\subset\R^3$ is called $\dih_{n}$\emph{-equivariant} for some $2\leq n\in\N$ if, for all $\varphi\in \dih_{n}$, the set $\varphi(E)$ coincides either with $E$ or with $\R^3\setminus E$ up to a negligible set (cf.~\cite[Definition 3.3]{Carlotto2020}). 
The following statement is a special case of \cite[Lemma 4.7]{Cianchi2011}. 
It uses the common notation $E\symdiff L\vcentcolon=(E\setminus L)\cup(L\setminus E)$ for the symmetric difference of sets. 

\begin{lemma}[Stability of the Gaußian isoperimetric inequality]\label{lem:isoperimetric}
Fix $2\leq n\in\N$. 
For every $\varepsilon>0$ there exists $\delta>0$ such that if $E\subset\R^3$ is a $\dih_{n}$-equivariant set with Gaußian measure $\Gamma(E)=\frac{1}{2}$ and perimeter $P(E)<1+\delta$ then there exists a $\dih_{n}$-equivariant half-space $L\subset\R^3$ such that $\Gamma(E\symdiff L)<\varepsilon$.   
\end{lemma}

\begin{proof}
We specialise the proof of \cite[Lemma 4.7]{Cianchi2011} to the equivariant setting. 
Suppose there exists $\varepsilon>0$ and a sequence $\{E_j\}_{j\in\N}$ of $\dih_n$-equivariant sets satisfying $\Gamma(E_j)=\frac{1}{2}$ and $P(E_j)\leq1+\frac{1}{j}$ for every $j\in\N$ as well as $\Gamma(E_j\symdiff L)\geq\varepsilon$ for every $\dih_{n}$-equivariant half-space $L\subset\R^3$. 
By a compactness argument, there exists a $\dih_n$-equivariant measurable set $E_{\infty}\subset\R^3$ such that up to a subsequence $\Gamma(E_{\infty}\symdiff E_j)\to0$ as $j\to\infty$.
In fact, the compactness theory from \cite[§\,4.2]{Cianchi2011} also applies to the class of \emph{equivariant} sets of finite Gaußian perimeter.
In particular, $\Gamma(E_{\infty})=\frac{1}{2}$ and by lower semicontinuity of the Gaußian perimeter, $P(E_{\infty})\leq1$.
Hence, $E_{\infty}$ is optimal in the Gaußian isoperimetric inequality \cite[Theorem 4.1]{Cianchi2011} and therefore equivalent to a half-space which contradicts our assumption. 
\end{proof}

\begin{lemma}[{cf.~\cite[Lemma 3.4]{Carlotto2020}}]\label{lem:perimeter}
For every $\{\Lambda_t\}_{t\in[0,1]}\in\Pi$, there exists a family $\{E_t\}_{t\in[0,1]}$ of $\dih_{g+1}$-equivariant sets $E_t\subset\R^3$ with finite Gaußian perimeter such that the following properties hold.
\begin{enumerate} [label={\normalfont(\roman*)}]
\item\label{lem:perimeter-i} $E_0$ coincides with the upper half-space and $E_1$ with the lower half-space.
\item\label{lem:perimeter-ii} $\Gamma(E_t\symdiff E_{t_0})\to0$ whenever $t\to t_0$.
\item\label{lem:perimeter-iii} $\Lambda_t=\partial E_t$ for all $t\in[0,1]$.  
\item\label{lem:perimeter-iv} $\Gamma(E_t)=\frac{1}{2}$ for all $t\in[0,1]$. 
\end{enumerate}
\end{lemma}

\begin{proof}
For our explicit model sweepout $\{\Sigma_t\}_{t\in[0,1]}$ constructed in Lemma \ref{lem:sweepout} it is straightforward to define a family $\{E_t^{\Sigma}\}_{t\in[0,1]}$ of $\dih_{g+1}$-equivariant sets with finite Gaußian perimeter satisfying properties \ref{lem:perimeter-i}--\ref{lem:perimeter-iv}. 
Since any other sweepout $\{\Lambda_t\}_{t\in[0,1]}\in\Pi$ allows a $\dih_{g+1}$-isotopy $\Phi\colon[0,1]\times\R^3\to\R^3$ such that $\Lambda_t=\Phi(t,\Sigma_t)$ we may define $E_t\vcentcolon=\Phi(t,E^{\Sigma}_t)$. 
Then, properties \ref{lem:perimeter-i}--\ref{lem:perimeter-iii} are evident. 
Considering the rotation $\psi_\ell$ by angle $\pi$ around the horizontal axis $\xi_\ell$ as given in \eqref{eqn:axes} which is an isometry of Gaußian space, it is easy to verify that $\psi_\ell(E_t)$ coincides with $\R^3\setminus E_t$ up to a negligible set. 
Therefore, $\Gamma(E_t)=\Gamma(\R^3\setminus E_t)=\frac{1}{2}$ which proves \ref{lem:perimeter-iv}. 
\end{proof}

\begin{lemma}[Width estimate]\label{lem:width}
There exists $\delta>0$ such that $1+\delta\leq W\leq 1.02+\frac{4}{e}<3$.
\end{lemma}

\begin{proof}
The upper bound follows directly from Lemma \ref{lem:sweepout} by definition of width. 
Let $\delta>0$ be the value associated to $\varepsilon=\frac{1}{8}$ in Lemma \ref{lem:isoperimetric}.   
Towards a contradiction, suppose that there exists a sweepout $\{\Lambda_t\}_{t\in[0,1]}\in\Pi$ such that 
\begin{align}\label{eqn:20200502}
\sup_{t\in[0,1]}F(\Lambda_t)<1+\delta.
\end{align}
Let $\{E_t\}_{t\in[0,1]}$ be the family of finite perimeter sets associated to $\{\Lambda_t\}_{t\in[0,1]}$ as given in Lemma~\ref{lem:perimeter}. 
Then, Lemma \ref{lem:isoperimetric} applies and yields that for every $t\in[0,1]$ there exists a $\dih_{g+1}$-equivariant half-space $L_t$ such that $\Gamma(E_t\symdiff L_t)<\frac{1}{8}$. 
By the triangle inequality for symmetric differences 
\begin{align}\label{eqn:triangle} 
\Gamma(L_{t}\symdiff L_{s})
\leq\Gamma(E_{t}\symdiff E_{s})+\tfrac{2}{8}
\end{align}
for any $t,s\in[0,1]$. 
The only $\dih_{g+1}$-equivariant half-spaces are the upper and lower half-spaces 
$\{x\in\R^3\mid \pm x_3>0\}$
and for $g=1$ also the half-spaces 
$\{x\in\R^3\mid \pm x_1>0\}$ and $\{x\in\R^3\mid \pm x_2>0\}$. 
In any case, \eqref{eqn:triangle} and property \ref{lem:perimeter-ii} imply that the map $t\mapsto L_t$ is constant. 
This however contradicts property \ref{lem:perimeter-i} and the claim follows. 
\end{proof}

Lemma \ref{lem:width} implies that the sweepout we have constructed satisfies the mountain pass condition $W>1=\max\{F(\Sigma_0),F(\Sigma_1)\}$.
Hence, by \cite[Theorem 2.4]{Ketover2018} and \cite[Theorem 1.3]{Ketoverc} there exists a min-max sequence $\{\Sigma^j\}_{j\in\N}$ of surfaces which are invariant under the action of $\dih_{g+1}$ and converge in the sense of varifolds to $m\Theta_g$, where $\Theta_g$ is a smooth, embedded minimal surface in Gaußian space $(\R^3,e^{-\frac{1}{4}\abs{x}^2}g_{\R^3})$. 
Moreover, we obtain the following properties (cf.~\cite[Remark 4.1]{Carlotto2020}). 
\begin{enumerate}[label={\normalfont(\roman*)}]
\item\label{min-max-theorem-axes} The surface $\Theta_g$ contains the horizontal axes $\xi_1,\ldots,\xi_{g+1}$ defined in \eqref{eqn:axes} and intersects the vertical axis $\xi_0$ defined in \eqref{eqn:axis0} orthogonally. 
\item\label{min-max-theorem-odd} The multiplicity $m$ is an odd integer.
\item\label{min-max-theorem-genus} $\genus(\Theta_g)\leq g$.
\end{enumerate}

Property \ref{min-max-theorem-axes} is visualised in Section \ref{sec:visualisation}, Figure \ref{fig:half-end} and follows from the corresponding property of the sweepout constructed in Lemma~\ref{lem:sweepout}.  
In particular, $\Theta_g$ is noncompact with at least one end. 
This is nontrivial since Gaußian space allows closed minimal surfaces like the sphere or Angenent's \cite{Angenent1992} torus. 
 
\begin{lemma}\label{lem:multiplicity}
The multiplicity $m$ in property \ref{min-max-theorem-odd} is equal to $1$ and $\Theta_g$ is not isometric to a flat plane. 
\end{lemma}

\begin{proof}
Convergence of the min-max sequence in the sense of varifolds to $m\Theta_g$ implies  
\begin{align}\label{eqn:20200412}
m F(\Theta_g)=W<3
\end{align}
where we recall the upper bound on the width $W$ from Lemma \ref{lem:width}.  
The Gaußian isoperimetric inequality implies $F(\Theta_g)\geq1$. 
Hence, estimate~\eqref{eqn:20200412} yields $m<3$ and with property~\ref{min-max-theorem-odd} we obtain $m=1$ as claimed. 
Consequently, the lower bound $F(\Theta_g)=W>1$ from Lemma~\ref{lem:width} implies that $\Theta_g$ is not isometric to the flat plane $\Sigma_0$. 
\end{proof}

\subsection{Determining the genus}\label{subsec:genus}

A \emph{regular cone} in $\R^3$ is a subset of the form $\mathcal{C}=\{\tau\mathcal{L}\mid\tau>0\}$, where $\mathcal{L}$ is a smooth, embedded,  codimension-one submanifold of the sphere $\Sp^2$. 
Moreover, we introduce the notation $\mathcal{Z}_v=\{x\in\R^3\mid\dist(x,\operatorname{span}(v))=\sqrt{2}\}$ for a self-shrinking cylinder in direction $v\in\R^3\setminus\{0\}$. 
By the following result of L.\;Wang \cite{Wang2016}, any end of our noncompact self-shrinker $\Theta_g$ must be asymptotic to a regular cone or a self-shrinking cylinder.  

\begin{theorem}[{L.~Wang \cite[Theorem 1.1]{Wang2016}}]\label{thm:cones}  
If $M$ is an end of a noncompact self-shrinker in $\mathbb R^3$ of finite topology then either of the following holds:
\begin{enumerate}[label={\normalfont(\roman*)}]
\item $\displaystyle\lim_{\tau\to\infty}\tau^{-1}M=\mathcal{C}$  in $C^\infty_{\mathrm{loc}}(\R^3\setminus\{0\})$ for a regular cone $\mathcal{C}$ in $\R^3$,
\item $\displaystyle\lim_{\tau\rightarrow\infty}(M-\tau v)=\mathcal{Z}_v$ 
 in $C^\infty_{\mathrm{loc}}(\R^3)$ for some $v\in\R^3\times\{0\}$.
\end{enumerate}
\end{theorem} 

\pagebreak[3]

Sun and Z.~Wang \cite[Appendix A]{Sun2020} observed that L.~Wang's argument for Theorem \ref{thm:cones} still works if the assumption of finite topology is replaced by finite genus. 
In particular, the following (much weaker) statement holds. 

\begin{corollary}\label{cor:connected}
For every $g\in\N$ there exists some radius $r_g>0$ such that for every $R\geq r_g$ every connected component of $\Theta_g\setminus\B_{R}$ is diffeomorphic to $\Sp^1\times\Interval{0,\infty}$ and such that $\Theta_g\cap\B_{R}$ is connected and properly embedded in $\B_{R}$ with $\genus(\Theta_g\cap\B_{R})=\genus(\Theta_g)$. 
\end{corollary}

\begin{proof}
As stated in \cite[Lemma A.2]{Sun2020}, the self-shrinker $\Theta_g$ can only have finitely many ends since its genus and Gaußian area are bounded. 
Thus, the existence of $r_g>0$ such that for every $R\geq r_g$ every connected component of $\Theta_g\setminus\B_{R}$ is diffeomorphic to $\Sp^1\times\Interval{0,\infty}$ is a consequence of Theorem \ref{thm:cones}. 
In particular, every connected component of $\Theta_g\setminus\B_{R}$ has only one boundary component. 
Therefore, if $\Theta_g\cap\B_{R}$ were disconnected, then $\Theta_g$ would be disconnected. This however contradicts the Frankel property of Gaußian space (cf. \cite{Impera2021}). 
\end{proof}

For the proof of Theorem \ref{thm:main} it remains to rule out that genus is lost in the min-max procedure. 
Before we prove the corresponding statements in Lemma \ref{lem:genus_Theta}, we study the general structure of equivariant surfaces.  
Given $2\leq n\in\N$, let $\Z_{n}<\dih_{n}$ be the cyclic subgroup of order $n$ generated by the rotation of angle $\frac{2\pi}{n}$ around the vertical axis $\xi_0\subset\R^3$.   

\begin{lemma}\label{lem:genus}
Given $1\leq g\in\N$, let $\Sigma\subset\R^3$ be any closed, connected, embedded $\Z_{g+1}$-equivariant surface of genus $\gamma\in\{1,\ldots,g\}$. 
If $\Sigma$ is disjoint from the vertical axis $\xi_0$ then $\gamma=1$. 
If $\Sigma$ intersects $\xi_0$ then the number of intersections is $4$ and $\gamma=g$. 
\end{lemma}

\begin{proof}
Being closed and embedded, $\Sigma\subset\R^3$ is two-sided and thus allows a global unit normal vector field $\nu$. 
Moreover, since $\Sigma$ is $\Z_{g+1}$-equivariant, $\nu$ is invariant under the action of $\Z_{g+1}$. 
If $\Sigma\cap\xi_0\ni x_0$ is nonempty, then $\nu(x_0)=\pm(0,0,1)$ since $x_0$ is fixed under rotation around $\xi_0$. 
Hence, any intersection of $\Sigma$ with the vertical axis $\xi_0$ must be orthogonal and the number $i$ of such intersections is well-defined, finite and even, that is $i=2j$ for some nonnegative integer $j$.  

The quotient $\Sigma'=\Sigma/\Z_{g+1}$ is an orientable topological surface without boundary and therefore has Euler-Characteristic $\chi(\Sigma')=2-2\gamma'$ for some nonnegative integer $\gamma'$. 
A variant of the Riemann--Hurwitz formula (see \cite[§\,IV.3]{Freitag2011} and \cite[Remark B.2]{Carlotto2020}) implies
\begin{align} 
\chi(\Sigma)&=(g+1)\chi(\Sigma')-2jg. 
\intertext{Substituting $\chi(\Sigma)=2-2\gamma$ and $\chi(\Sigma')=2-2\gamma'$, we obtain}
\label{eqn:20210301}
\gamma&=(g+1)\gamma' +(j-1)g.
\end{align}
All variables in equation \eqref{eqn:20210301} are nonnegative integers.
If $\Sigma$ is disjoint from $\xi_0$, then $j=0$ and the assumption $1\leq\gamma\leq g$ implies $\gamma'=1$ and hence $\gamma=1$. 
The case $j=1$ would imply $\gamma=0$ or $\gamma\geq(g+1)$ which contradicts our assumption on $\gamma$. Similarly, $j\geq3$ is excluded. 
The only remaining case is $j=2$ which then implies $\gamma'=0$, $i=4$ and $\gamma=g$ as claimed. 
\end{proof}

The following corollary generalises \cite[Lemma B.1]{Carlotto2020} by dropping the assumption that the boundary of the surface in question is connected.

\begin{corollary}\label{cor:genus}
Given $1\leq g\in\N$, let $\B\subset\R^3$ be any convex, bounded, $\Z_{g+1}$-equivariant domain with piecewise smooth boundary and let $\Sigma\subset\B$ be any compact, connected, $\Z_{g+1}$-equivariant surface of genus $\gamma\in\{1,\ldots,g\}$ which is properly embedded, i.\,e.~$\partial\Sigma=\Sigma\cap\partial B$.  
Then $\Sigma$ has genus $\gamma\in\{1,g\}$ and if the intersection with the axis of rotation is nonempty, then $\Sigma$ has genus $\gamma=g$.  
\end{corollary}

\begin{proof} 
The genus of $\Sigma$ is defined as the genus of the closed surface which is obtained by closing up each boundary component of $\Sigma$ with a topological disc. 
Below we will show that this procedure can be done while preserving the embeddedness and the $\Z_{g+1}$-equivariance of $\Sigma$. 
Since $B$ is convex, $\Sigma$ can be extended radially to a $\Z_{g+1}$-equivariant, properly embedded surface in some round ball and this extension does not change the genus of $\Sigma$. 
Hence, we may assume without loss of generality that $B$ is a ball around the origin.  

Let $\beta$ be any boundary component of $\Sigma$ and, of the two domains in $\partial B$ bounded by $\beta$, let $S_\beta$ be the one with smaller area (or any of the two domains if they have equal area). 
Given $R_\beta\geq1$, let $D_\beta\vcentcolon=[1,R_\beta]\beta\cup R_\beta S_\beta$, where the multiplication by $[1,R_\beta]$ respectively $R_\beta$ is to be understood as scaling. 
Then, $D_\beta$ is a topological disc. 
Choosing $R_\beta=1+\operatorname{area}(S_\beta)$ we can achieve that the corresponding discs $D_\beta$ are pairwise disjoint  as one varies $\beta$. 
Indeed, let $\beta_1$ and $\beta_2$ be two arbitrary boundary components of $\Sigma$ which we label such that $\operatorname{area}(S_{\beta_1})\leq\operatorname{area}(S_{\beta_2})$. 
Assuming $\beta_1\neq\beta_2$ we have $\beta_1\cap\beta_2=\emptyset$. 
If $S_{\beta_1}\cap S_{\beta_2}=\emptyset$ then $D_{\beta_1}\cap D_{\beta_2}=\emptyset$ follows trivially. 
If $S_{\beta_1}\cap S_{\beta_2}\neq\emptyset$ then $S_{\beta_1}\subset S_{\beta_2}$ and $R_{\beta_1}<R_{\beta_2}$ which implies $D_{\beta_1}\cap D_{\beta_2}=\emptyset$. 
Smoothing the union $\bigcup_\beta D_\beta\cup\Sigma$ equivariantly, we obtain a closed, embedded, $\Z_{g+1}$-equivariant surface. 
The claim then follows from Lemma~\ref{lem:genus}. 
\end{proof}

\begin{lemma}\label{lem:genus_Theta}
The genus of the self-shrinker $\Theta_g$ is equal to $g$. 
\end{lemma}

\begin{proof}
Let the radius $r_g>0$ be as in Corollary \ref{cor:connected}. 
If we prove that $\genus(\Theta_g)\geq1$, then the claim follows by applying Corollary \ref{cor:genus} to $\Sigma\vcentcolon=\Theta_g\cap\B_{r_g}$ since $\Sigma\subset\B_{r_g}$ is a connected, properly embedded, $\dih_{g+1}$-equivariant surface containing the origin with $\genus(\Sigma)\leq g$.   

Brendle's \cite{Brendle2016} classification result implies that if $\genus(\Theta_g)=0$ then $\Theta_g$ is the horizontal plane because it contains the origin. This case however is excluded by Lemma \ref{lem:multiplicity}. 
\end{proof}

\section{Behaviour for high genus} 
\label{sec:high_genus}

The examples of complete self-shrinkers with high genus in \cite{Kapouleas2018,Nguyen2014} were found by desingularising the intersection of a sphere with a plane. 
While it remains open whether these surfaces are isometric to our surfaces $\Theta_g$ for sufficiently large $g$, we prove that our family $\{\Theta_g\}_{g\in\N}$ has the same asymptotic behaviour, i.\,e.~that it converges in the sense of varifolds to the union of a sphere and a plane as $g\to\infty$ as stated in Theorem \ref{thm:high_genus}.

\begin{proof}[Proof of Theorem \ref{thm:high_genus}]
We consider any subsequence of $\{\Theta_g\}_{g\in\N}$ without relabelling and recall that the Gaußian area of $\Theta_g$ is bounded from above uniformly in $g$ by Lemma~\ref{lem:width}. 
Given any $R_0>2$, there exists a fixed radius $R\geq R_0$ such that $\Theta_g$ intersects the sphere $\partial\B_R$ of radius $R$ transversally for any $g\in\N$. 
This follows from Sard's Theorem. 
In order to avoid having to estimate the mass of the boundary $\partial(\Theta_g\cap\B_R)$, we close the surface $\Theta_g\cap\B_R$ by adding a suitable subset $Q_R\subset\partial\B_R$ such that $\Omega_g^R\vcentcolon=\Theta_g\cap\B_R\cup Q_R$ is a Lipschitz surface without boundary. 
This does not necessarily preserve the dihedral symmetry but it can be done preserving the invariance with respect to the cyclic subgroup $\Z_{g+1}$ as in the proof of Corollary~\ref{cor:genus}. 
We interpret $\Omega_g^R$ as integral varifold in the Riemannian manifold $(\B_{2R},e^{-\frac{1}{4}\abs{x}^2}g_{\R^3})$ with mass bounded by $3+R^2e^{-\frac{1}{4}R^2}$ uniformly in $g$ and with boundary mass equal to zero. 
Allard's \cite{Allard1972} compactness theorem implies the existence of a further subsequence which converges in the sense of varifolds to some stationary integer varifold $\Omega_\infty^{R}$ as $g\to\infty$. 
Then, we restrict the elements $\Omega_g^{R}$ of the subsequence and their limit $\Omega_\infty^{R}$ back to the interior of $\B_R$ which preserves the convergence. 
 
\emph{Claim 1}. The support of $\Omega_\infty^{R}$ contains the intersection of $\B_R$ with the horizontal plane 
\begin{align*}
P\vcentcolon=\{x\in\R^3\mid x\perp(0,0,1)\}.
\end{align*}
\begin{proof}[Proof of Claim 1.]
Let $\mu_{\Omega_\infty^{R}}$ be the weight measure of the varifold $\Omega_\infty^{R}$ and let $x_0\in P\cap\B_R$. 
If $x_0$ is not in the support of $\mu_{\Omega_\infty^{R}}$ then there exists $\varepsilon>0$ such that $\mu_{\Omega_\infty^{R}}(\B_\varepsilon(x_0))=0$. 
However, there exists $g_\varepsilon\in\N$ such that for all $g\geq g_\varepsilon$ the union $\xi_1\cup\ldots\cup\xi_{g+1}$ of the horizontal lines fixed by the dihedral group $\dih_{g+1}$ intersects $\B_{\varepsilon/2}(x_0)$. 
Since $\xi_1\cup\ldots\cup\xi_{g+1}\subset\Omega_g^{R}$ we obtain some $\delta_\varepsilon>0$ such that $\mu_{\Omega_g^{R}}(B_\varepsilon(x))\geq\delta_\varepsilon$ uniformly in $g$, a contradiction. 
\end{proof}

Let $x\in\B_{R}$ be arbitrary but fixed. 
In the following, we understand all geometric quantities with respect to the 
Gaußian ambient metric $e^{-\frac{1}{4}\abs{x}^2}g_{\R^3}$ restricted to $\B_R$. 
Let $A_{g}$ denote the second fundamental form of the minimal surface $\Omega_g^R\subset\B_R$. 
Given $x\in\B_R$, the $\varepsilon$-regularity theorem of Choi--Schoen 
\cite[Proposition 2]{Choi1985} states that there exists $\varepsilon_0>0$ such that if $r\leq\varepsilon_0$ and if
\begin{align*}
\int_{\Omega_g^{R}\cap\B_r(x)}\abs{A_{g}}^2\,d\mu\leq\varepsilon_0
\end{align*}
then there exists a constant $C$ which depends only on the background geometry and can hence be chosen uniformly with respect to the genus $g$ such that 
\begin{align}\label{eqn:Choi-Schoen}
\max_{0\leq\sigma\leq r}\sigma^2\sup_{\B_{r-\sigma}(x)}\abs{A_{g}}^2\leq C. 
\end{align}
Complementary to the set of points where this result applies, we consider (cf.~\cite[§\,4.3]{Ketovera})
\begin{align}\label{eqn:blow-up-set}
\Lambda^{R}\vcentcolon=\Bigl\{x\in\B_R\;\Big\vert\;\inf_{r>0}\Bigl(\liminf_{g\to\infty}\int_{\Omega_g^R\cap\B_r(x)}\abs{A_{g}}^2\,d\mu\Bigr)\geq\varepsilon_0\Bigr\}.
\end{align}
The equivariance of $\Omega_g^{R}$ implies that $\Lambda^{R}$ is rotationally symmetric around the vertical axis $\xi_0$.  
Since by construction, $\Omega_g^{R_1}\subseteq\Omega_g^{R_2}\subset\Theta_g$ for any $0<R_1\leq R_2$ and any $g\in\N$ we also have $\Lambda^{R_1}\subseteq\Lambda^{R_2}$. 

\emph{Claim 2.} 
Let $x_0\in\Lambda^{R}\setminus\xi_0$ be arbitrary. 
Then the genus of $\Omega_g^R$ restricted to any neighbourhood of $x_0$ is unbounded as $g\to\infty$. 

\begin{proof}[Proof of Claim 2.]
Let the point $x_0\in\Lambda^{R}\setminus\xi_0$, the radius $r_0>0$ and the constant $\gamma\in\N$ be given. 
By Sard's Theorem, there exists $0<r<r_0$ such that $\partial\B_r(x_0)$ intersects $\Omega_g^R$ transversally for all $g$. 
In particular, the genus of $\Omega_g^R\cap\B_r(x_0)$ is well-defined. 
Given some $\gamma,g_0\in\N$, suppose $\genus(\Omega_g^R\cap\B_r(x_0))\leq\gamma$ for all $g\geq g_0$. 
Then, since area and mean curvature of $\Omega_g^R\cap\B_r(x_0)$ are bounded uniformly in $g$, Ilmanen's \cite[Lecture 3]{Ilmanen1998} localised Gauß--Bonnet estimate implies  
\begin{align}\label{eqn:20210304-1}
\sup_{g\geq g_0}\int_{\Omega_g^R\cap\B_{r/2}(x_0)}\abs{A_{g}}^2\,d\mu\leq C.
\end{align}
Let $S$ be the horizontal circle around $\xi_0$ through $x_0$. 
Then $S\subset\Lambda^{R}$ since $\Lambda^{R}$ is rotationally symmetric.   
Given any $n\in\N$ there exists a collection of points $x_1,\ldots,x_n\in S\cap\B_{r/2}(x_0)$ with pairwise distance $>2\delta>0$, 
depending only on $r$, $n$ and the radius of $S$. 
By definition \eqref{eqn:blow-up-set},
\begin{align}\label{eqn:20210304-2}
\sup_{g\geq g_0}\int_{\Omega_g^R\cap\B_{r/2}(x_0)}\abs{A_{g}}^2\,d\mu
\geq\sup_{g\geq g_0}\sum^{n}_{k=1}
\int_{\Omega_g^R\cap\B_{\delta}(x_k)}\abs{A_{g}}^2\,d\mu\geq \frac{n\varepsilon_0}{2}. 
\end{align}
Since $n\in\N$ is arbitrary, the estimates \eqref{eqn:20210304-1} and \eqref{eqn:20210304-2} are in contradiction and we obtain that there exists a subsequence $g\to\infty$ along which $\genus(\Omega_g^R\cap\B_r(x_0))\to\infty$. 
\end{proof}

\emph{Claim 3.} $\Lambda^{R}$ is contained in $P\cup\xi_0$, where $P$ denotes the horizontal plane and $\xi_0$ the vertical axis. 
Moreover, $\Lambda^{R}\cap P\setminus\xi_0$ consists of at most one circle around $\xi_0$. 
 
\begin{proof}[Proof of Claim 3.]
Towards a contradiction, suppose that there exists $x_0\in\Lambda^{R}\setminus(P\cup\xi_0)$. 
Due to the dihedral symmetry, we may assume that $x_0$ lies above $P$. 
Given $0<\varepsilon<\frac{1}{2}\dist(x_0,P\cup\xi_0)$ there exists $1\ll g\in\N$ (which we now fix) such that at least one connected component $Q_{g,\varepsilon}$ of $\Omega_g^R\cap\B_\varepsilon(x_0)$ satisfies $\genus(Q_{g,\varepsilon})\geq1$ by Claim 2. 
Let $O_{g,\varepsilon}$ be the orbit of $Q_{g,\varepsilon}$ under the action of the cyclic subgroup $\Z_{g+1}<\dih_{g+1}$. 
Then, $O_{g,\varepsilon}$ is either connected or $O_{g,\varepsilon}$ has $g+1$ connected components which are all isometric to $Q_{g,\varepsilon}$. 
The fact that $O_{g,\varepsilon}\subset\Omega_g^{R}$ can have at most genus $g$ excludes the second case.  
Hence, $O_{g,\varepsilon}$ is connected and since it contains at least two disjoint copies of $Q_{g,\varepsilon}$ by construction it has at least genus $2$. 

Let $0<r<\varepsilon$ such that the boundary of the domain $Y_r^R\vcentcolon=\{(y_1,y_2,y_3)\in\B_R\mid y_3\geq r\}$ intersects $\Omega_g^R$ transversally for all $g$ and let $\Sigma$ be the connected component of $\Omega_g^R\cap Y_r^R$ which contains $O_{g,\varepsilon}$. 
Since $\Sigma$ is $\Z_{g+1}$-equivariant and properly embedded in the convex domain $Y_r^R$, Corollary~\ref{cor:genus} applies and yields $\genus(\Sigma)=g$. 
However, since $\Sigma\subset\Omega_g^R$ is contained in the upper half space, the dihedral symmetry implies $\genus(\Omega_g^R)\geq2g$ in contradiction with $\genus(\Omega_g^R)\leq g$. 
This proves $\Lambda^{R}\subset P\cup\xi_0$. 

Being rotationally symmetric, $\Lambda^{R}\cap P$ is a union of circles. 
Suppose, $\Lambda^{R}\cap P\setminus\xi_0$ contains two circles with different radii $0<r_1<r_2<R$.  
Let $r\in\interval{r_1,r_2}$ such that $\Omega_g^R$ intersects the sphere $\partial\B_r$ of radius $r$ around the origin transversally. 
Then $\Omega_g^R\cap\B_r$ is properly embedded in $\B_r$ for all $g$, and, arguing similarly as above, it has a connected component  with full genus $g$. 
This implies $\genus(\Omega_g^R\setminus\B_r)=0$ in contradiction with Claim 2 applied to neighbourhoods of points in the outer circle. 
\end{proof}

\emph{Claim 4.} The limit $\Omega_{\infty}^{R}$ is smooth and embedded away from $\Lambda^{R}\cap P$ and $\Lambda^{R}\cap \xi_0$ is discrete. 

\begin{proof}[Proof of Claim 4.]
Let $x_0\in\B_R\setminus\Lambda^{R}$ be arbitrary. 
Then, as stated in \eqref{eqn:Choi-Schoen}, there exists $r>0$ such that the second fundamental form $A_{g}$ of $\Omega_g^{R}\cap\B_r(x_0)$ is bounded uniformly for all $g$ along a (further) subsequence. 
This implies (cf.~\cite{Langer1985}) that the convergence $\Omega_g^{R}\to \Omega_\infty^{R}$ is smooth in $\B_r(x_0)$.

Now let $x_0\in\Lambda^{R}\cap\xi_0\setminus P$ be arbitrary. 
Then there exist $0<r<\dist(x_0,P)$ and $g_0\in\N$ such that any connected component of $\Omega_g^{R}\cap\B_r(x_0)$ has genus $0$ or $1$ for all $g\geq g_0$. Otherwise, Corollary \ref{cor:genus} would imply $\genus(\Omega_g^{R}\cap\B_r(x_0))\geq g$ and hence $\genus(\Omega_g^{R})\geq2g$ by dihedral symmetry.  
Once the genus is uniformly bounded, White's result \cite[Theorem 1.1]{White2018} implies that the limit $\Omega_\infty^{R}\cap\B_r(x_0)$ is smooth and that the convergence is smooth away from a discrete set. 
\end{proof}

To summarise, there is a subsequence along which the convergence $\Omega_g^R\to\Omega_\infty^{R}$ is smooth away from $\Lambda^R$ and $\Lambda^R$ consists of at most one circle in the horizontal plane $P$ in addition to a discrete subset of $\xi_0$. 
Moreover, $\Omega_\infty^{R}$ is smooth away from $\Lambda^R\cap P$.  
If $\Lambda^{R_0}\cap P$ is not empty for some $R_0$, then 
$\Lambda^{R}\cap P=\Lambda^{R_0}\cap P=c$ for all $R\geq R_0$ since the set $\Lambda^{R}$ is monotone with respect to $R$ in the sense of inclusions. 
Otherwise, we set $c=\emptyset$. 
A diagonal argument leads to a further subsequence along which our complete self-shrinkers $\Theta_g$ converge locally smoothly away from $c$ to some limit $\Theta_\infty$ as $g\to\infty$. 
Moreover, $\Theta_\infty$ has the same symmetries as $\Omega_\infty^R$ and is embedded away from $c$. 

\emph{Claim 5.} 
If $R>0$ is sufficiently large then $c=\Lambda^{R}\cap P$ is nonempty with positive diameter and $\Theta_\infty$ is not just a multiple of the plane $P$. 

\begin{proof}[Proof of Claim 5.]
Towards a contradiction, suppose that $\Lambda^{R}\cap P\subset\{0\}$ for all $R$. 
Then Claim 3 implies that $\Lambda^{R}$ is a discrete set contained in the vertical axis $\xi_0$ for all $R$. 
Hence, $\Theta_\infty\setminus\{0\}$ is a smooth, embedded minimal surface in Gaußian space. 
Being rotationally symmetric, each connected component of $\Theta_\infty \setminus P$ that has $0$ in its closure has finite Euler characteristic. 
Hence, by the removable singularities result \cite{Gulliver1976} (see also \cite[Prop. 1]{Choi1985}), each of these components extends smoothly to the origin, lies on one side of $P$ (because $\Theta_\infty$ is embedded away from the origin) and touches $P$ only in the origin. 
This is impossible by the strong maximum principle, and therefore no such components exist.
By Claim 1 and the Frankel property, $\Theta_\infty$ must therefore be a multiple of the plane $P$.
The dihedral symmetry implies that the multiplicity $m$ must be odd. 
In fact $m=1$ since the Gaußian area of $\Theta_g$ is bounded from above by $1.02+\frac{4}{e}<3$ for all $g\in\N$ by Lemma~\ref{lem:width}. 

By Brakke's \cite{Brakke1978} work (cf.~\cite[p.\;54]{Colding2013}), the plane $P$ is the self-shrinker with least Gaußian area and there is a gap to the next lowest. 
Hence, if a subsequence of $\{\Theta_g\}_g$ converges locally smoothly to $P$ with multiplicity one, then $\Theta_g=P$ for sufficiently large $g$ which contradicts the fact that $\Theta_g$ has genus $g$ as shown in Lemma~\ref{lem:genus_Theta}.  
\end{proof}

\emph{Claim 6.} 
The blow-up of $\Theta_\infty$ around any point $x_0\in c$ is given by two orthogonal planes. 
In particular, $\Theta_\infty$ is the union of the plane $P$ with a smooth self-shrinker $S$. 

\begin{proof}[Proof of Claim 6]
Let $T_r\vcentcolon=\{x\in\R^3\mid\dist(x,c)<r\}$ be the tubular neighbourhood of radius $r>0$ around $c$. 
We recall that $\Theta_\infty$ is rotationally symmetric. 
For all sufficiently small $r>0$, the restriction $\Theta_\infty\cap T_r$ is a union of finitely many rotationally symmetric annuli which share $c$ as one of their boundary curves and whose other boundary curve lies on $\partial T_r$.  
For any $0<\varepsilon<r$ the convergence 
$\Theta_g\cap (T_r\setminus T_\varepsilon)\to\Theta_\infty\cap (T_r\setminus T_\varepsilon)$ is smooth as $g\to\infty$ along our subsequence.  
The dihedral symmetry and the fact that $P\subsetneq\Theta_{\infty}$ imply that $\Theta_\infty\cap (T_r\setminus T_\varepsilon)$ has at least $4$ connected components. 
We claim that there are exactly $4$ such components. 
 
In order to make Corollary \ref{cor:genus} applicable, we introduce the convex hull $\tilde{T_r}$ of $T_r$.  
As in the proof of Claim 3 we can show that for all sufficiently large $g$, the restricted surface $\Theta_g\cap\tilde{T_r}$ has one connected component $Q_g$  which has full genus $g$. 
Moreover, Lemma \ref{lem:genus} implies that once all the boundary components of $Q_g$ are closed up $\Z_{g+1}$-equivariantly by topological discs (cf.~proof of Corollary~\ref{cor:genus}), the resulting surface intersects the vertical axis $\xi_0$ exactly $4$ times. 
One of these intersections is located in the origin; thus $Q_g$ has at most $3$ boundary components on $\partial\tilde{T}_r$ which wind around $\xi_0$. 
If $\Theta_g\cap\tilde{T}_r$ has another connected component $Q_g'$ for infinitely many $g$ along our subsequence, then necessarily $\genus(Q_g')=0$ and \cite[Theorem 1.1]{White2018} implies that $\{Q_g'\}_{g}$ converges smoothly as $g\to\infty$ away from a discrete set. 
In particular, the limit cannot be singular along $c$. 
This however contradicts our choice of $r$. 
Consequently, $\Theta_\infty\cap\tilde{T}_r$ also has at most $3$ boundary components on $\partial\tilde{T}_r$ which wind around $\xi_0$. 
Therefore, $\Theta_\infty\cap (T_r\setminus T_\varepsilon)$ has at most $4$ and hence exactly $4$ connected components. 

It remains to prove the orthogonality along $c$.  
We recall that $\Theta_\infty$ is a stationary varifold in Gaußian space whose support is given by the union $P\cup S$, where $P$ is the horizontal plane and where the connected components of $S\setminus P$ are minimal surfaces of Gaußian space with boundary $c$. 
The blow-up of $\Theta_\infty$ around $x_0\in c$ is a stationary varifold $W$ in Euclidean space. 
Since $\Theta_\infty$ is rotationally symmetric and invariant under the rotation of angle $\pi$ around the axis through $x_0$ and the origin, we have that $W=X\times\R$.
The argument above shows that the profile $X$ consists of exactly $4$ rays $\xi_+,\zeta_+,\xi_-,\zeta_-$ emerging from the origin $x_0$. 
The rays $\xi_{\pm}$ corresponding to the horizontal plane $P$ form a straight line $\xi$ as shown in Figure~\ref{fig:blowup}. 
Stationarity implies that the configuration must be balanced, i.\,e.~the union of the remaining two rays $\zeta_{\pm}$ must again form a straight line $\zeta$. 
Since $X$ is symmetric with respect to $\xi$, the intersection of $\xi$ with $\zeta$ must be orthogonal. 
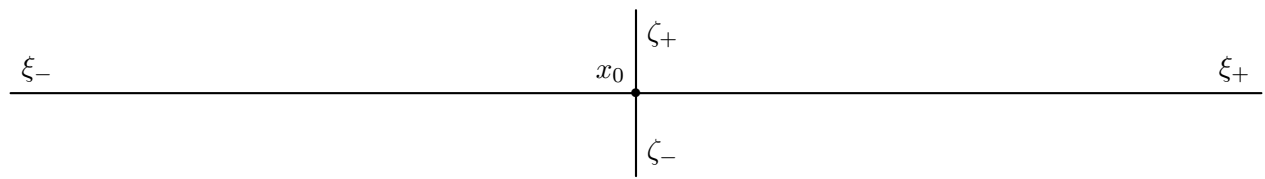
\begin{figure}\centering
\begin{tikzpicture}[line cap=round,line join=round,baseline={(0,0)},thick]
\pgfmathsetmacro{\xmax}{\textwidth/2.005cm}
\pgfmathsetmacro{\ymax}{0.95}
\draw(-\xmax,0)node[above right]{$\xi_-$}--(\xmax,0)node[above left]{$\xi_+$};
\draw(0,0)node{$\scriptstyle\bullet$}node[above left]{$x_0$};
\draw(0,-\ymax)node[above right,inner sep=0]{~$\zeta_-$}--(0,\ymax)node[below right,inner sep=0]{~$\zeta_+$};
\end{tikzpicture}
\caption{Notation for the profile $X$ of the blow-up around $x_0$.}%
\label{fig:blowup}%
\end{figure}
Therefore, the intersection of $S$ with $P$ is orthogonal. 
Hence, $S$ is of class $C^1$ across $P$. 
Standard regularity theory yields that $S$ is in fact a smooth, embedded self-shrinker without boundary. 
\end{proof}

\emph{Conclusion.} 
Claim 6 implies that $\Theta_\infty=P\cup S$, where $S$ is a smooth, embedded, rotationally symmetric self-shrinker intersecting the plane $P$ orthogonally in $c$. 
Hence, the profile $\sigma\vcentcolon=S\cap\{x_2=0\leq x_1\}$ is a smooth, connected curve intersecting the $x_1$-axis exactly once. 
In particular, $\sigma$ can not be closed. Therefore, $S$ has genus zero. 
By Brendle's \cite{Brendle2016} classification result, $S$ must be a (self-shrinking) sphere or cylinder since the horizontal plane with multiplicity has already been excluded in Claim 5. 
The upper Gaußian area bound in Lemma~\ref{lem:width} implies $F(P\cup S)\leq 1.02+\frac{4}{e}$, which is slightly below the Gaußian area $1+\sqrt{2\pi/e}$ of the union of $P$ with a self-shrinking cylinder (cf.~Table~\ref{table:entropy}). 
Consequently, $S$ is a self-shrinking sphere and $\Theta_\infty$ is the union of a plane and a sphere as claimed. 
This completes the proof since the initial subsequence of $\{\Theta_g\}_{g\in\N}$ was chosen arbitrarily.
\end{proof}

\section{Visualisations and numerical results}\label{sec:visualisation}

In this section, we display visualisations of self-shrinkers $\Theta$ for mean curvature flow in $\R^3$. 
Each image is the result of a numerical simulation with Brakke's \cite{Brakke1992} surface evolver. 
If $\Theta$ is noncompact, then we can only simulate a compact section of the surface in which case we restrict the simulation to the ball $\B_8$ of radius $8$ around the origin. 
The simulation starts from a compact, triangulated surface $\Sigma$ with the expected symmetry and topology. 
In the case that $\Theta$ is noncompact, its approximation $\Sigma\subset\B_8$ has nonempty boundary on $\partial\B_8$, where we impose a free boundary condition. 
Then we use Brakke's program to evolve $\Sigma$ towards a minimal surface with respect to the Gaußian background metric. 
The free boundary condition, i.\,e.~the requirement that the resulting surface meets the sphere $\partial\B_8$ orthogonally, is natural in the setting of self-shrinkers with asymptotically conical ends. 
Maintaining a good triangulation during the evolution is delicate because triangles far from the origin have extremely small Gaußian area and tend to degenerate. 
We designed a custom made algorithm to keep the triangulation homogeneous in order to obtain ``smooth'' visualisations. 
The algorithm directly outputs approximate values for the Gaußian area of the simulated surfaces. 
We collect some of these values in Table \ref{table:entropy} together with the Gaußian areas of the self-shrinkers with genus zero as a reference. 
The Gaußian area of Angenent's \cite{Angenent1992} torus has been computed numerically in \cite{Berchenko-Kogan2021} and this value is slightly larger than what we obtain for $F(\Theta_1)$. 

\begin{figure} 
\begin{minipage}{0.49\textwidth}
\includegraphics[width=\textwidth]{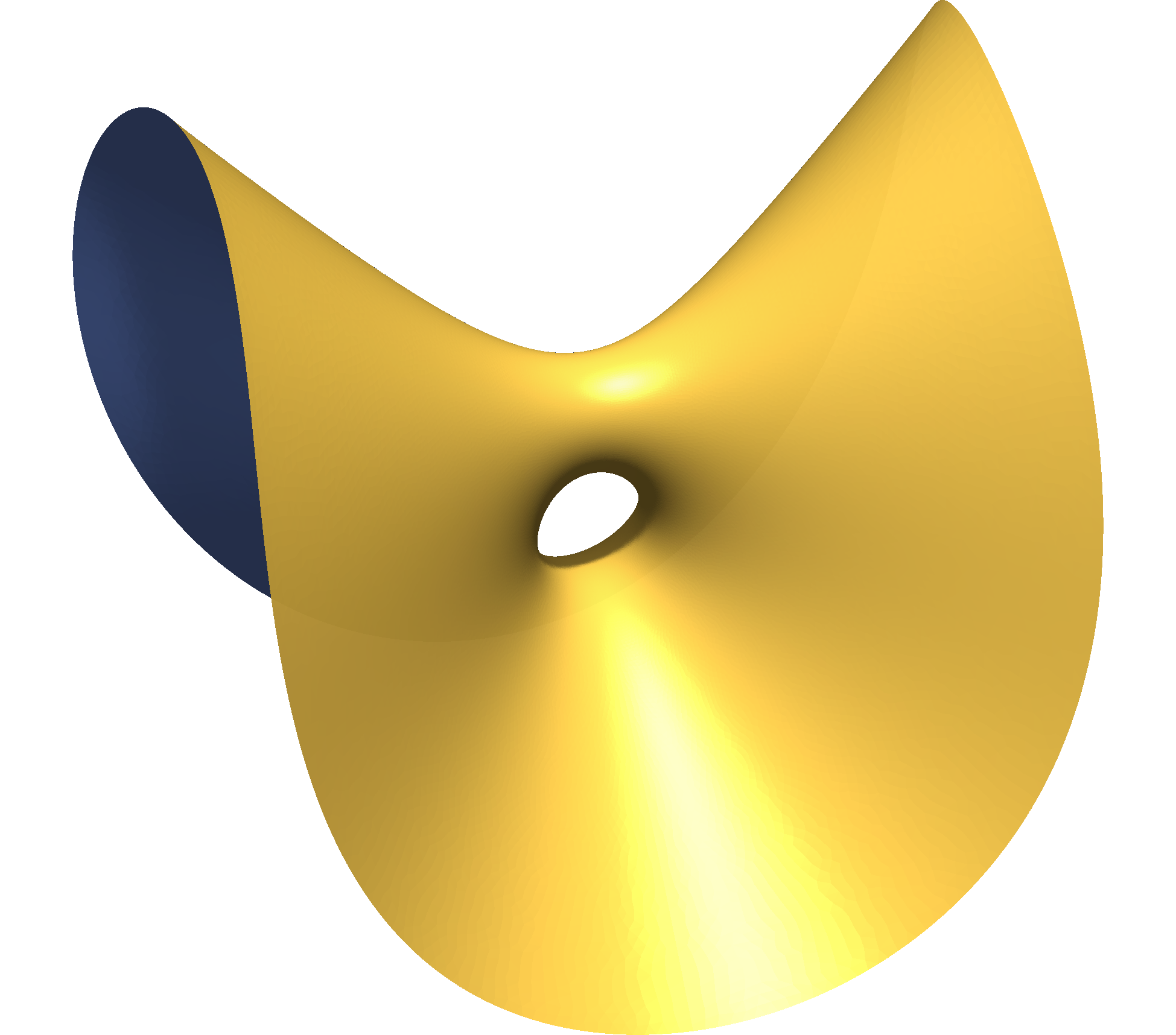}
\end{minipage}
\hfill
\begin{minipage}{0.49\textwidth}
\includegraphics[width=\textwidth]{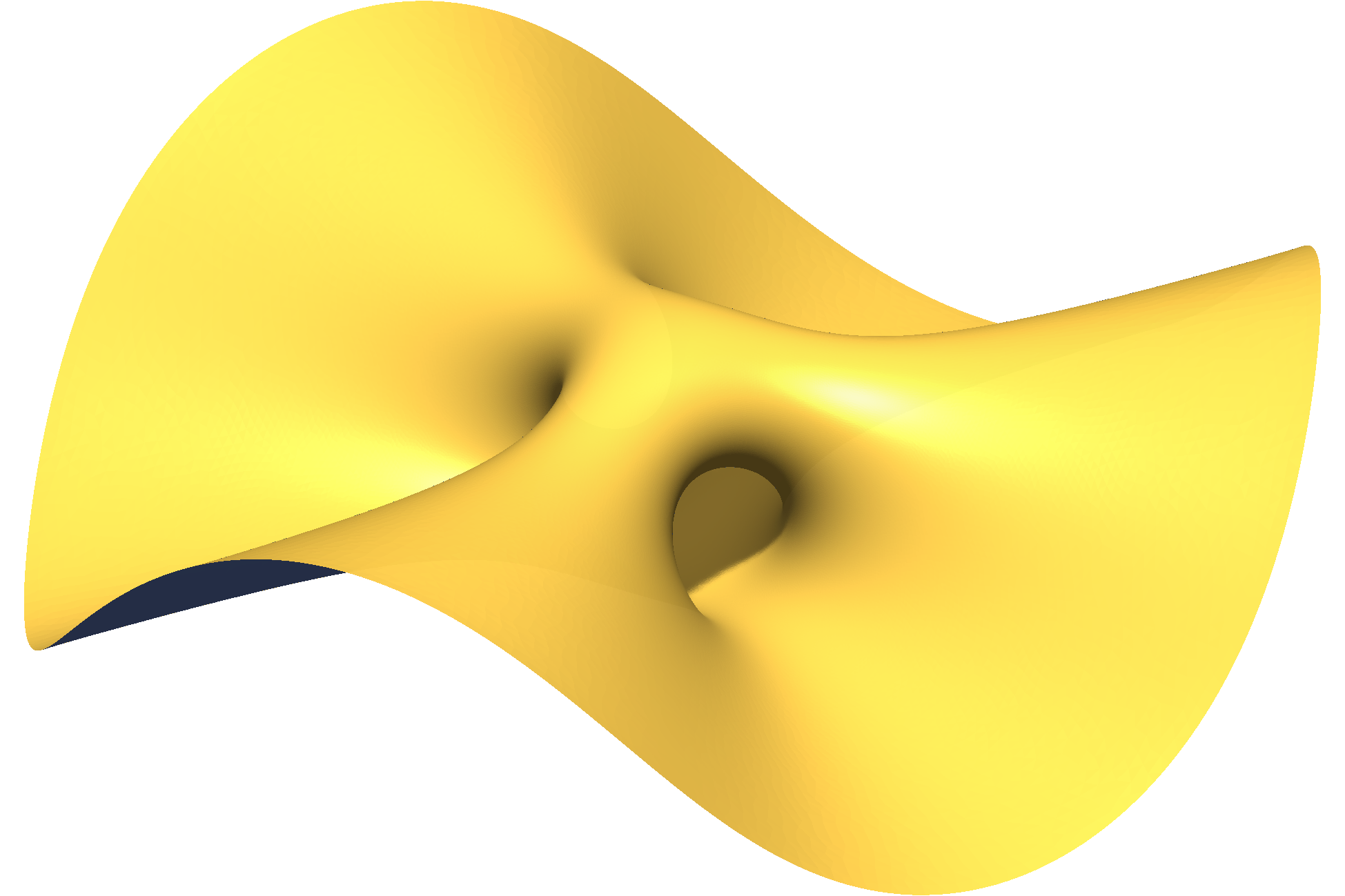}
\\ 
\includegraphics[width=\textwidth]{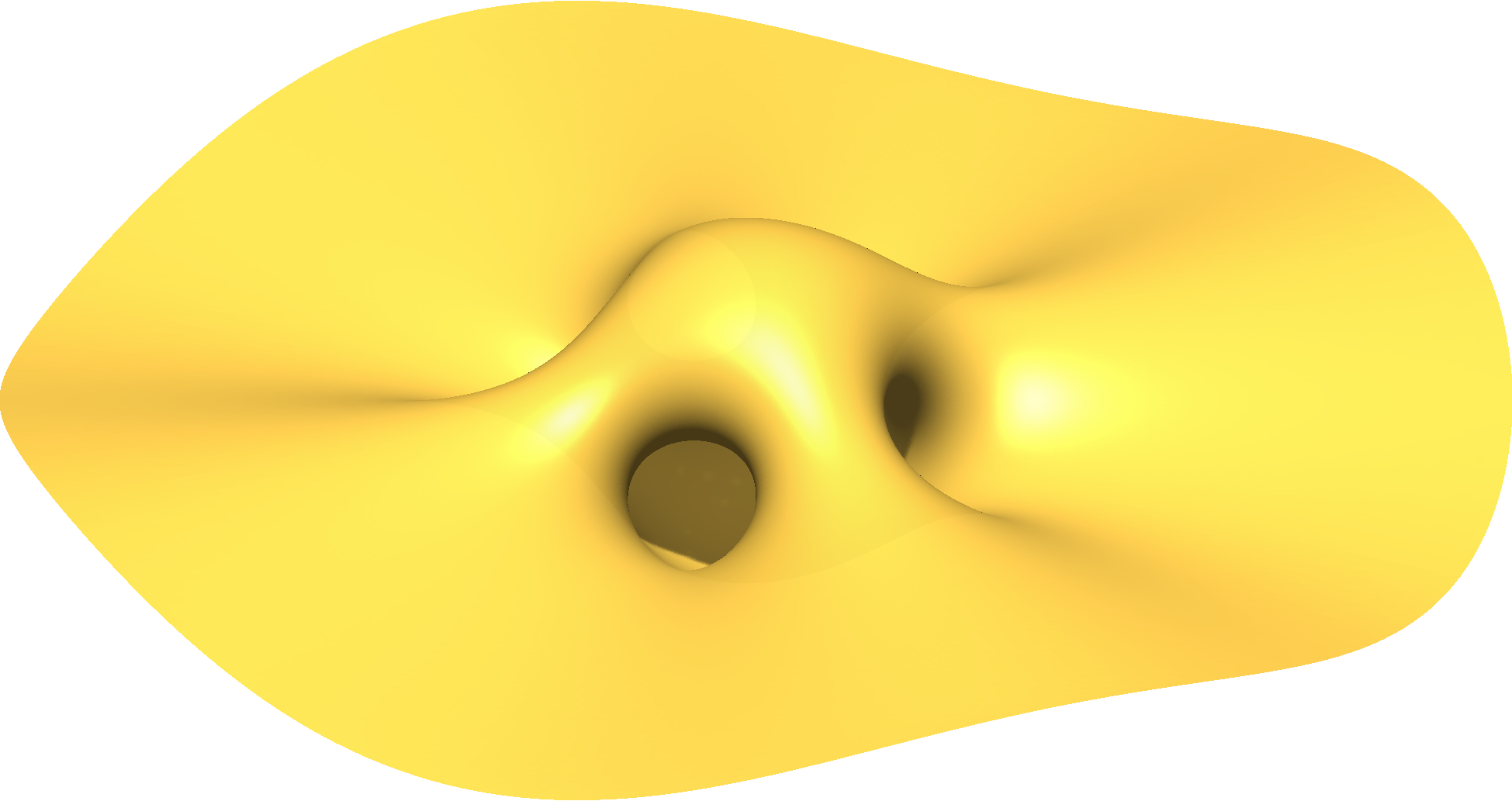}
\end{minipage}
\caption{Self-shrinkers of genus $g\in\{1,2,3\}$ with one end.} 
\label{fig:one-end}%
\end{figure}  	

Figures \ref{fig:one-end} and \ref{fig:half-end} visualise the expected shapes of the self-shrinkers $\Theta_g$ constructed in Theorem~\ref{thm:main}.  
The simulated surfaces feature reflection symmetries along vertical planes in addition to their dihedral symmetry.  
This indicates that the full symmetry group of $\Theta_g$ is actually larger than the dihedral group $\dih_{g+1}$ as remarked at the beginning of Section~\ref{sec:existence}.   
In Figure \ref{fig:2ends} (top images), we visualise the expected shapes of the self-shrinkers $\Phi_{2n-1}$ from Conjecture~\ref{conj:twoends} for $n\in\{7,24\}$. 
We compare them to simulations of the compact self-shrinkers $\Sigma_{2n}$ (bottom images) which M\o{}ller \cite{Moeller2014} constructed for sufficiently large genus $g=2n$ by desingularising the union of Angenent's torus and the sphere.  
We observe that $\Phi_{2n-1}$ and $\Sigma_{2n}$ are both invariant under the action of the dihedral group $\dih_n$. 
Moreover, if the outermost intersection of $\Sigma_{2n}$ with the horizontal plane is moved to infinity, then the resulting surface is isotopic to $\Phi_{2n-1}$. 
This illustrates that if one would attempt to construct either $\Phi_{2n-1}$ or $\Sigma_{2n}$ via equivariant min-max methods, it would be especially delicate to control the topology of the resulting surface.

\clearpage

\begin{figure}[t!] 
\begin{minipage}{0.5333\textwidth}
\includegraphics[width=\textwidth]{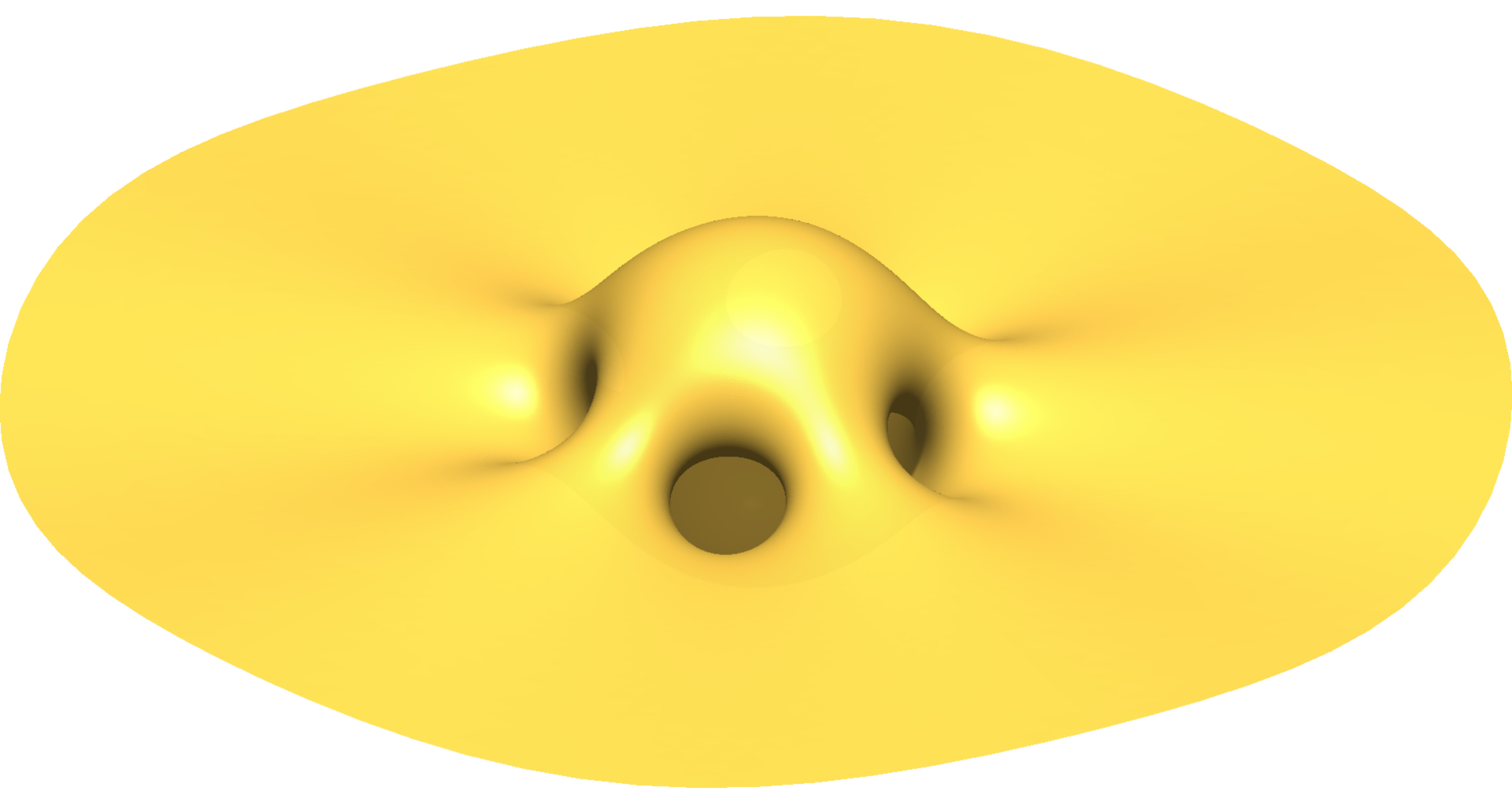}\\[1ex] 
\includegraphics[width=\textwidth]{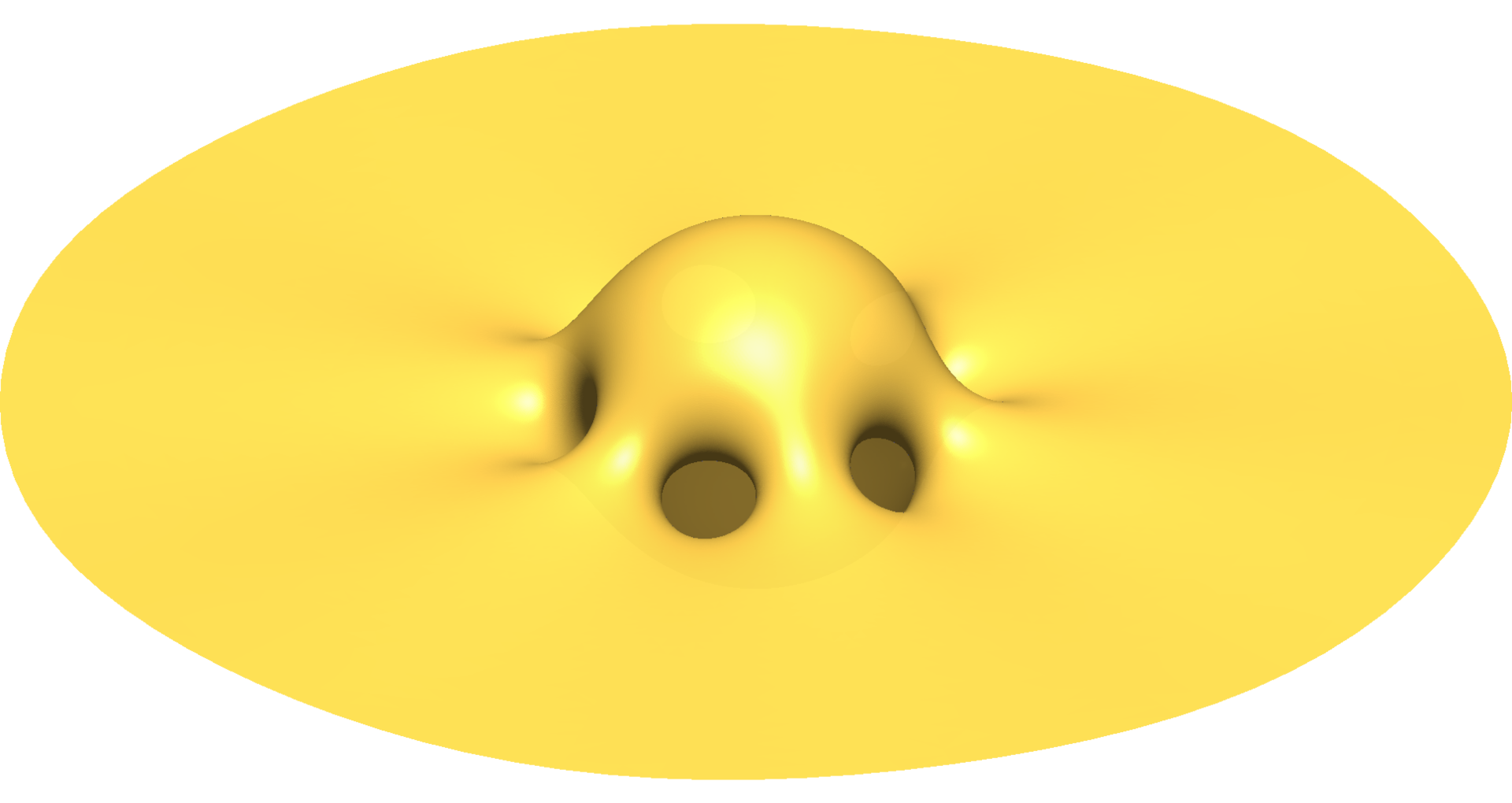}\\[1ex] 
\includegraphics[width=\textwidth]{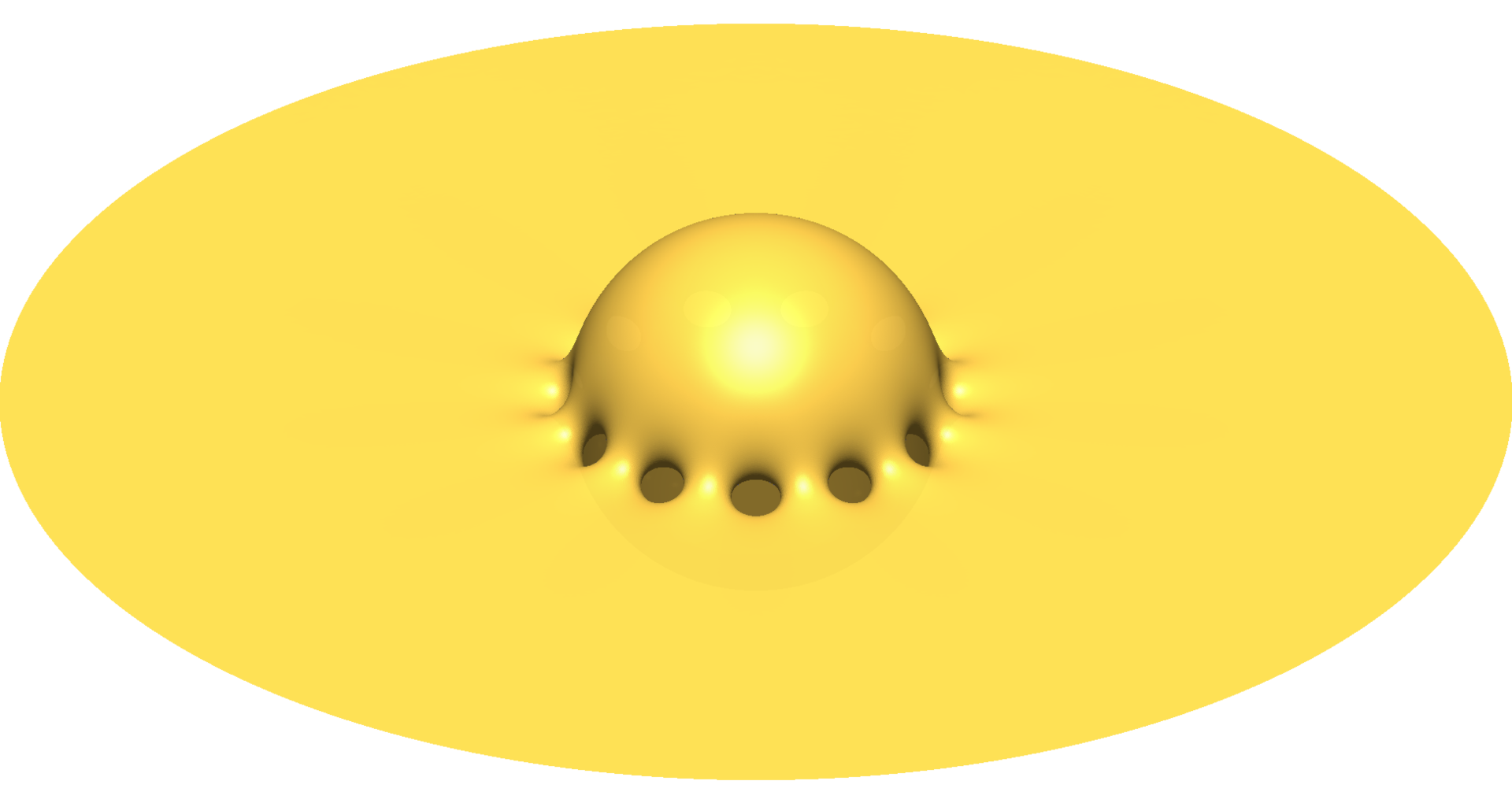} 
\end{minipage}
\hfill
\begin{minipage}{0.4583\textwidth}
\includegraphics[width=\textwidth]{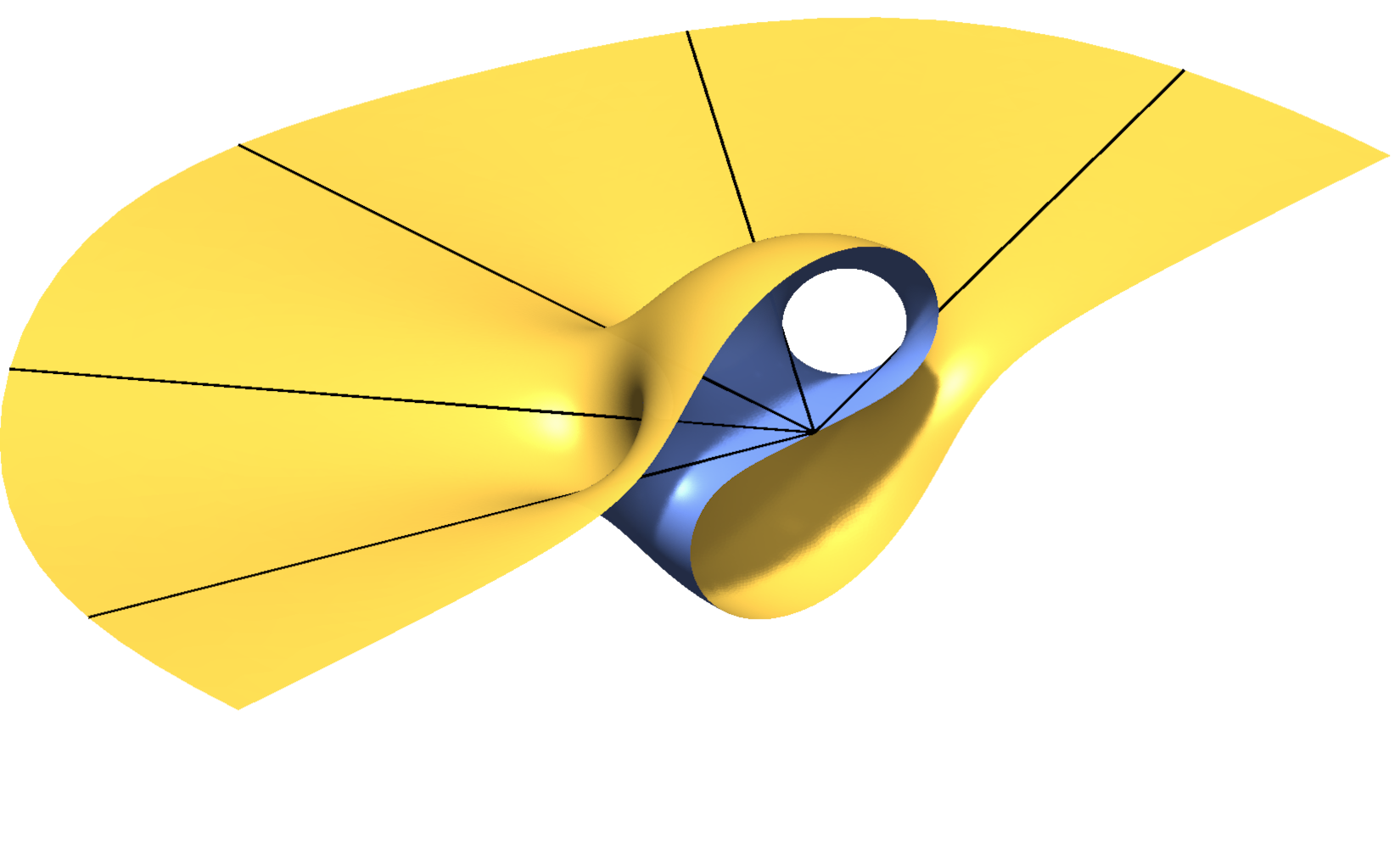}\\[1ex]
\includegraphics[width=\textwidth]{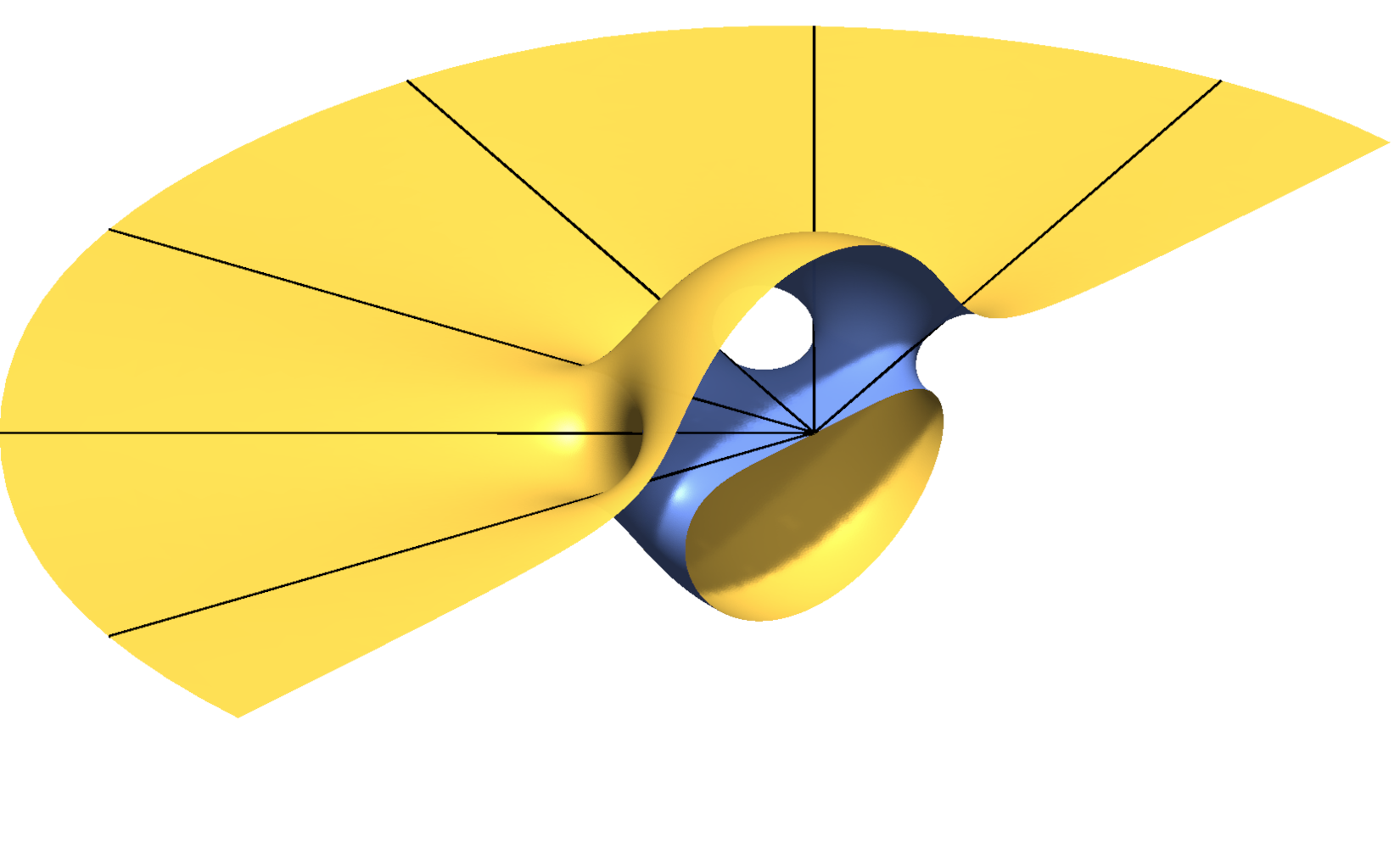}\\[1ex] 
\includegraphics[width=\textwidth]{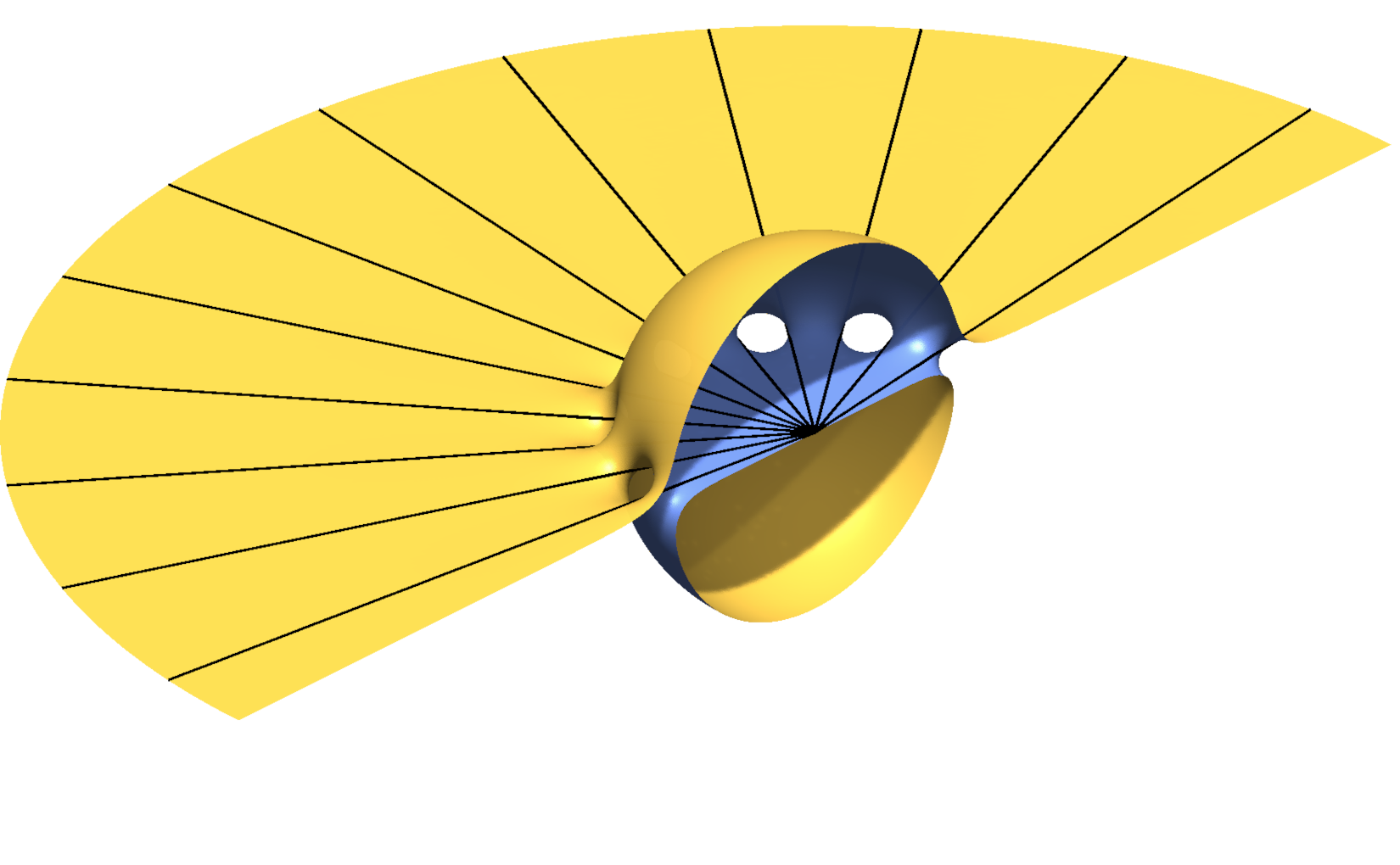} 
\end{minipage}
\caption{Self-shrinkers of genus $g\in\{4,5,11\}$ with one end and their vertical cuts containing the horizontal axes $\xi_1,\ldots,\xi_{g+1}$.  
} 
\label{fig:half-end}%
\end{figure} 

\begin{table} 
\caption{Approximate values for the Gaußian area of certain self-shrinkers.}
\label{table:entropy}
\smallskip\centering
\renewcommand{\arraystretch}{1.2}%
\begin{tabular}{
>{\raggedleft\arraybackslash}p{0.25\textwidth-2\tabcolsep}|
>{\raggedleft\arraybackslash}p{0.25\textwidth-2\tabcolsep}
}
{self-shrinker in $\R^3$} & {Gaußian area $F(\cdot)$}  
\\\hline\hline 
flat plane  & $1\hphantom{.0000}$ 
\\\hline
sphere of radius $2$ & $4/e\approx 1.4715$
\\\hline
cylinder of radius $\sqrt{2}$ &  $\sqrt{2\pi/e}\approx 1.5203$
\\\hline
$\mathrlap{\Theta_1}\qquad$ & $1.7143$
\\\hline
Angenent's torus & $1.8512$ 
\\\hline
$\mathrlap{\Theta_2}\qquad$ & $2.0433$
\\\hline
$\mathrlap{\Theta_3}\qquad$ & $2.1817$
\\\hline
$\mathrlap{\Theta_{11}}\qquad$ & $2.3854$
\end{tabular}
\end{table}

\clearpage

\begin{figure} 
\centering
\includegraphics[width=0.499\textwidth]{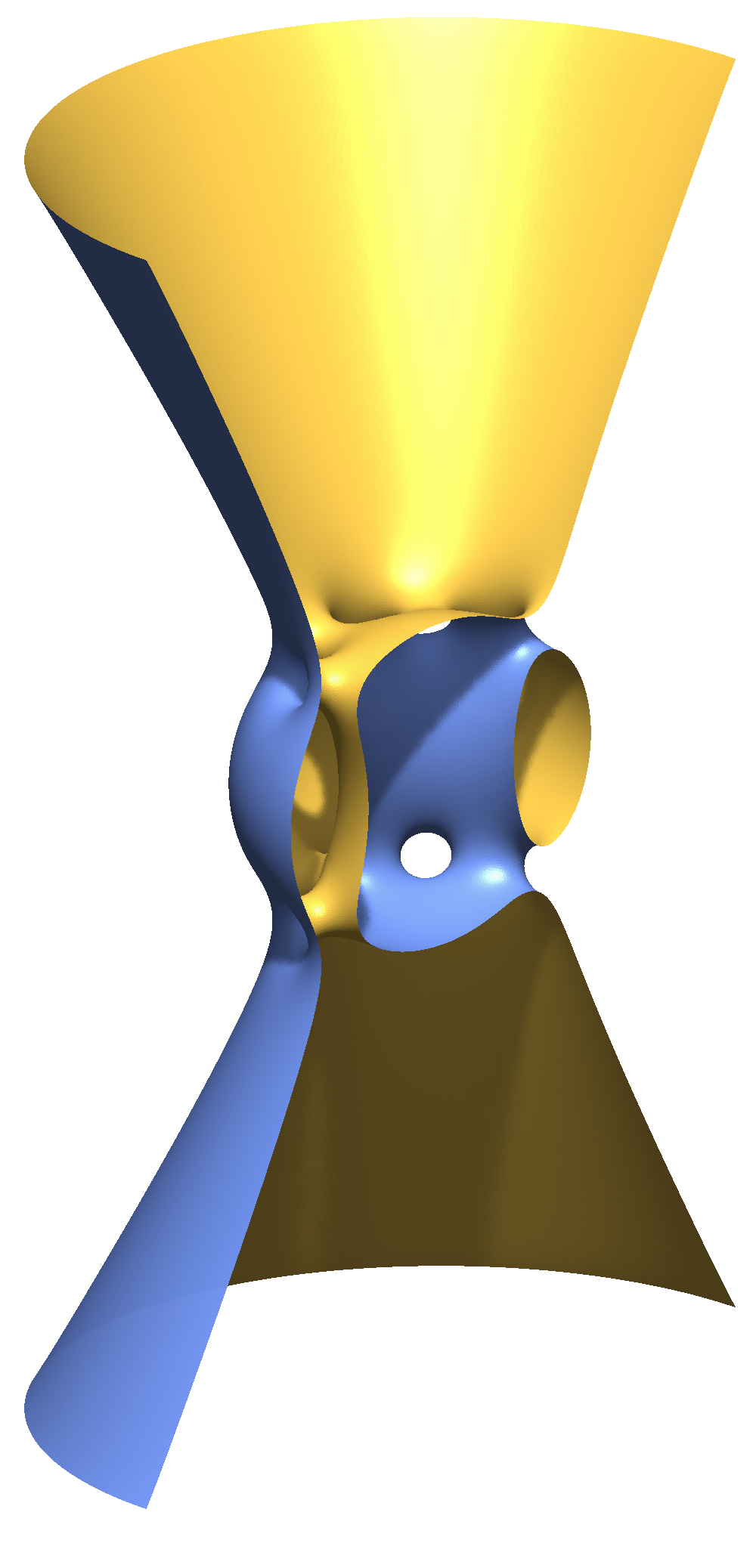}%
\hfill%
\includegraphics[width=0.499\textwidth]{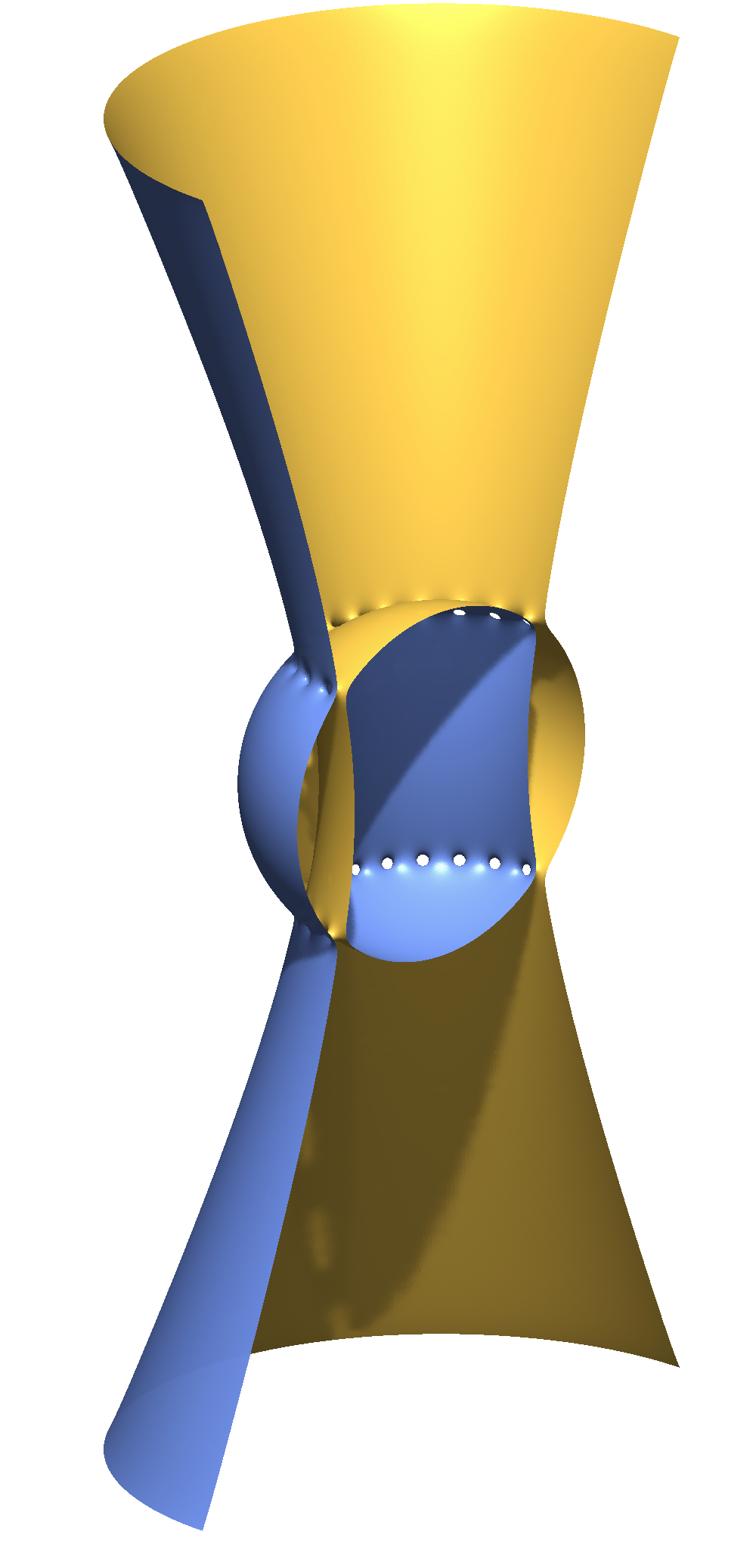}%
\\[-4ex]
\includegraphics[width=0.499\textwidth]{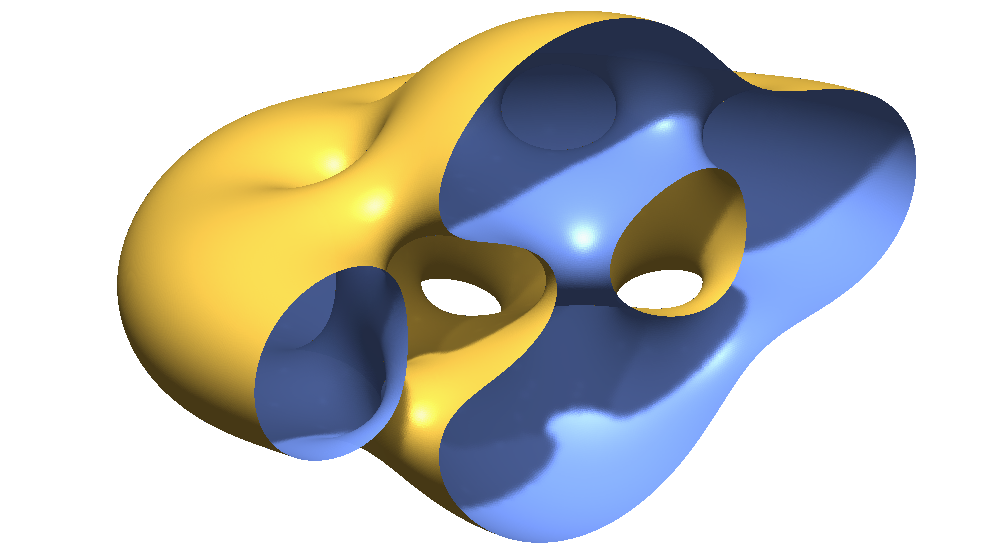}%
\hfill%
\includegraphics[width=0.499\textwidth]{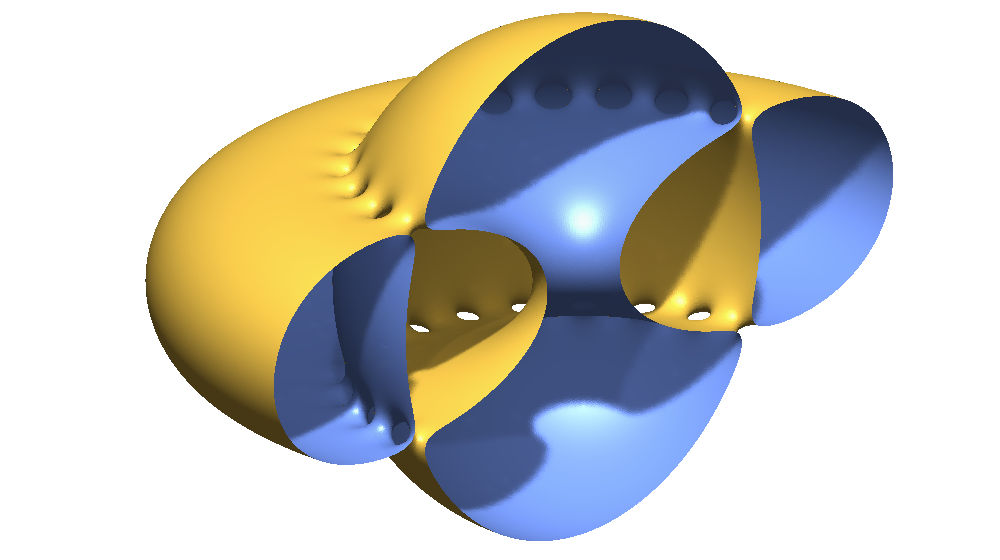}%
\caption{Top: Vertical cut through self-shrinkers of genus $13$ and $47$ with two ends. \newline
Bottom: Vertical cut through compact self-shrinkers of genus $14$ and $48$.}
\label{fig:2ends}%
\end{figure}


\clearpage
 
\setlength{\parskip}{1ex plus 1pt minus 1pt}
\bibliography{MCF-bibtex}

\printaddress

\end{document}